\documentclass{amsart}

\RequirePackage{color}

\def\smallskip{\vskip\smallskipamount}
\def\medskip{\vskip\medskipamount}
\def\bigskip{\vskip\bigskipamount}

\usepackage{graphicx}
\usepackage{amsmath,amsthm,bm,mathrsfs,amscd,tikz}
\usepackage{ytableau}
\usepackage{comment}
\usepackage{mathdots}
\usepackage[all]{xy}
\usepackage{graphicx}
\usepackage[colorinlistoftodos]{todonotes}
\usepackage{enumerate}
\usepackage{feynmp-auto}
\usepackage[english]{babel}
\usepackage[utf8]{inputenc}
\usepackage{float}
\usepackage{tikz}
\usepackage{tikz-cd}
\usepackage{amssymb}
\usepackage{graphicx}
\usepackage{mathtools}
\usetikzlibrary{decorations.pathmorphing}

\newtheoremstyle{thmstyle}{}{}{\itshape}{}{\bfseries}{ }{5pt}{}
\newtheoremstyle{exstyle}{}{}{}{}{\bfseries}{ }{5pt}{}
\newtheoremstyle{defstyle}{}{}{}{}{\bfseries}{ }{5pt}{}
\newtheoremstyle{remstyle}{}{}{}{}{\bfseries}{ }{5pt}{}

\theoremstyle{thmstyle}
\newtheorem{thm}{Theorem}[section]
\newtheorem{theorem}[thm]{Theorem}

\newtheorem{lemma}[thm]{Lemma}

\newtheorem{proposition}[thm]{Proposition}

\newtheorem{corollary}[thm]{Corollary}

\theoremstyle{exstyle}

\newtheorem{example}[thm]{Example}

\theoremstyle{defstyle}

\newtheorem{definition}[thm]{Definition}
\newtheorem{def-prop}[thm]{Definition-Proposition}

\theoremstyle{remstyle}

\newtheorem{remark}[thm]{Remark}

\theoremstyle{remstyle}

\newcommand{\Str}{\operatorname{Str}}

\newcommand{\Hom}{\operatorname{Hom}}

\newcommand{\Ext}{\operatorname{Ext}}

\DeclareMathOperator*{\modu}{mod}

\DeclareMathOperator*{\add}{add}

\DeclareMathOperator*{\Brick}{Brick}
\DeclareMathOperator*{\ind}{ind}
\DeclareMathOperator*{\tors}{tors}
\DeclareMathOperator*{\ftors}{f-tors}
\DeclareMathOperator*{\torf}{torf}
\DeclareMathOperator*{\wide}{wide}
\DeclareMathOperator*{\fwide}{f-wide}
\DeclareMathOperator*{\End}{End}
\DeclareMathOperator*{\Bic}{Bic}
\DeclareMathOperator*{\Sim}{Sim}
\DeclareMathOperator{\JI}{JI}
\DeclareMathOperator*{\torshad}{torshad}
\DeclareMathOperator*{\widshad}{widshad}
\DeclareMathOperator*{\str}{str}
\DeclareMathOperator*{\Gen}{gen}
\DeclareMathOperator*{\Filt}{filt}

\newcommand{\bic}{\mathcal{B}ic}
\newcommand{\simeqd}{\mathrel{\rotatebox[origin=c]{-90}{$\simeq$}}}

\def\ra{\rightarrow}
\def\Mcal{\mathcal{M}}
\DeclareMathOperator{\Cov}{Cov}
\def\pr{\prime}\def\ov{\overline}
\newcommand*{\vcenteredhbox}[1]{\begingroup
\setbox0=\hbox{#1}\parbox{\wd0}{\box0}\endgroup}

\begin{document}


\title{A categorification of biclosed sets of strings}

\author[Garver]
{Alexander Garver}
\address{LaCIM, UQAM, Montréal, Québec, Canada}
\email{alexander.garver@gmail.com}

\author[McConville]{Thomas McConville}
\address{Massachusetts Institute of Technology, Cambridge, MA, USA}
\email{thomasmcconvillea@gmail.com}

\author[Mousavand]{Kaveh Mousavand} 
\address{LaCIM, UQAM, Montréal, Québec, Canada}
\email{mousavand.kaveh@gmail.com }

\thanks{AG received support from the Canada Research Chairs Program and NSERC grant RGPIN/05999-2014. KM received Bourse de stage international du FRQNT for a 4-month research stay at MIT and was also partially supported by ISM Scholarship.}


\maketitle

\begin{abstract}
We consider the closure space on the set of strings of a gentle algebra of finite representation type. Palu, Pilaud, and Plamondon proved that the collection of all biclosed sets of strings forms a lattice, and moreover, that this lattice is congruence-uniform. Many interesting examples of finite congruence-uniform lattices may be represented as the lattice of torsion classes of an associative algebra. We introduce a generalization, the lattice of torsion shadows, and we prove that the lattice of biclosed sets of strings is isomorphic to a lattice of torsion shadows.

Finite congruence-uniform lattices admit an alternate partial order known as the shard intersection order. In many cases, the shard intersection order of a congruence-uniform lattice is isomorphic to a lattice of wide subcategories of an associative algebra. Analogous to torsion shadows, we introduce wide shadows, and prove that the shard intersection order of the lattice of biclosed sets is isomorphic to a lattice of wide shadows.
\end{abstract}

\section{Introduction}

Let $\Lambda$ be a finite dimensional associative algebra over a field $k$, and let $\modu(\Lambda)$ be the category of finitely generated left modules over $\Lambda$. A \emph{torsion class} is a full, additive subcategory of $\modu(\Lambda)$ that is closed under quotients and extensions. We consider the collection $\tors(\Lambda)$ of all torsion classes of $\Lambda$ as a poset ordered by inclusion. The poset $\tors(\Lambda)$ is a complete lattice \cite[Proposition 2.3]{IRTT}. Moreover, the lattice of torsion classes is known to be semidistributive \cite{GM} and completely congruence-uniform \cite{DIRRT}. This additional lattice structure is interesting from an algebraic point of view since it encodes homological information of $\Lambda$ as order-theoretic information.

The purpose of this article is to introduce the notion of a \emph{torsion shadow}, which is defined as the intersection of a torsion class with some fixed subcategory $\mathcal{M}$ of $\modu(\Lambda)$. We are particularly interested in bound quiver algebras $\Lambda$ obtained by ``doubling'' a gentle quiver; see Section~\ref{String Algebra} for background on gentle algebras and Section~\ref{torshad_section} on the doubling construction. The category $\mathcal{M}$ is additively generated by a certain collection of string modules, also specified in Section~\ref{torshad_section}. Before stating our main results, we summarize our motivation as follows.

To study the structure of $\tors(\Lambda)$ for a certain family of Jacobian algebras, $\tors(\Lambda)$ was realized in \cite{GM} as a quotient of a lattice of biclosed sets. We say a subset $X$ of a closure space is a \emph{biclosed set} if both $X$ and its complement are closed. The archetypal family of biclosed sets are the inversion sets of permutations of $n$, which corresponds to a certain closure space on the $2$-element subsets of $\{1,\ldots,n\}$. Bj\"orner and Wachs \cite{bjorner.wachs:shellable_2} introduced a surjective function from permutations of $n$ to binary trees with $n$ nodes in the context of poset topology, which has since found significance in combinatorial Hopf algebras \cite{loday.ronco:1998hopf}, constructions of generalized associahedra \cite{hohlweg2011permutahedra}, cluster algebras \cite{ReadingCamb}, and many other areas. From \cite{thomas:tamari}, we may interpret their map as a lattice quotient map from biclosed sets to $\tors(kQ)$ where $Q$ is the path quiver with $n-1$ vertices. Similar maps from biclosed sets to torsion classes were presented in \cite{GM} and \cite{palu2017non}.


In \cite{GM}, a categorification of biclosed sets as \emph{biclosed subcategories} is given, which we recall in Section~\ref{sec:biclosed}. However, while the poset $\Bic(\Lambda)$ of biclosed subcategories of $\modu(\Lambda)$ is a graded, congruence-uniform lattice for the algebras $\Lambda$ appearing in that paper, it is not even a lattice for a general algebra $\Lambda$. Furthermore, the lattice structure of biclosed subcategories in \cite{GM} does not have a clear homological interpretation. The main motivation for this article is to correct these deficiencies by interpreting biclosed subcategories as the torsion shadows of another algebra.


We now describe our main results. Let $(Q,I)$ be a gentle bound quiver and $A=kQ/I$ a gentle algebra. We let $\Pi(A)=k\ov{Q}/\ov{I}$ be the algebra for the ``doubled'' quiver, as defined in Section~\ref{torshad_section}. Then there is a canonical surjective homomorphism $\Pi(A)\ra A$ inducing a lattice map on torsion classes $\tors(\Pi(A))\ra \tors(A)$. Let $\mathcal{M}$ be the set of string modules that are reorientations of strings in $\modu(A)$. Then the lattice map factors as
\[ \tors(\Pi(A)) \ra \torshad{}_{\mathcal{M}}(\Pi(A)) \ra \tors(A). \]

\begin{theorem}\label{thm_main_1}
  There is an isomorphism of lattices $\Bic(A) \cong \torshad{}_{\mathcal{M}}(\Pi(A))$, which identifies the map $\Bic(A)\ra\tors(A)$ with $\torshad_{\mathcal{M}}(\Pi(A))\ra\tors(A)$.
\end{theorem}

Analogously to torsion shadows, we introduce the notion of a \emph{wide shadow} in Section~\ref{sec:wide}, which is the intersection of a wide subcategory of $\modu(\Lambda)$ with a distinguished subcategory $\Mcal$. For the algebra $\Pi(A)$ and our choice of $\Mcal$, we exhibit a correspondence between wide shadows and torsion shadows that mimics the usual correspondence between wide subcategories and torsion classes given in \cite{IT},\cite{MS}.

\begin{theorem}\label{thm_main_2}
  There is a natural bijection between $\widshad(\Pi(A))$ and $\torshad(\Pi(A))$.
\end{theorem}

We construct the bijection in Theorem~\ref{thm_main_2} in two ways -- first by using maps that resemble the ones defined in \cite{MS}, and secondly by identifying $\widshad(\Pi(A))$ with the lattice-theoretic shard intersection order of the (finite) congruence-uniform lattice $\torshad(\Pi(A))$. From the second description, we obtain a large family of congruence-uniform lattices whose shard intersection orders are also lattices, which is not true for general congruence-uniform lattices; see \cite{muhle} and \cite[Problem 9.5]{ReadingPAB}. The shard intersection orders we consider in this work include those discovered in \cite{cliftondillerygarver} where the lattice property was also proved.

The rest of the paper is organized as follows. Background on lattices and representations of gentle algebras is given in Sections~\ref{sec:lattice_prelim},\ref{sec:representation_prelim}, and~\ref{sec:gentle_alg_prelim}. The lattice structure of biclosed sets of strings is examined in Section~\ref{sec:biclosed}. Torsion shadows are introduced in Section~\ref{torshad_section} and Theorem~\ref{thm_main_1} is proved. Wide shadows are introduced in Section~\ref{sec:wide}. The canonical join complex and shard intersection order of the lattice of biclosed sets is determined in Sections~\ref{sec:canonical} and~\ref{sec:shard}, culminating in a proof of Theorem~\ref{thm_main_2}.



\section*{Acknowledgements}

The authors thank Hugh Thomas for helpful ideas and useful conversations early on in the project. The third author would like to profoundly thank the Department of Mathematics, as well as the International Students Office, at MIT for their unflagging support during his four-month research visit that resulted in this project. 

\section{Lattice theory preliminaries}\label{sec:lattice_prelim}


We recall some background on lattices. Proofs of claims made in this section may be found in \cite{freese1995free} and \cite[Section 2]{garver2016oriented}.

Let $(L,\le_L)$ be a finite lattice. For $x,y \in L$, if $x < y$ and there does not exist $z \in L$ such that $x < z <y$, we write $x \lessdot y$. Let $\Cov(L) := \{(x,y) \in L^{2} \mid x \lessdot y\}$ be the set of \textit{covering relations} of $L$. We let $\hat{0}, \hat{1} \in L$ denote the unique minimal and unique maximal elements of $L$, respectively. 


We say that an element $j \in L$ is \textit{join-irreducible} if $j \neq \hat{0}$ and whenever $j = x \vee y$, one has that $j = x$ or $j = y$. \textit{Meet-irreducible} elements $m \in L$ are defined dually. We denote the subset of join-irreducible (resp., meet-irreducible) elements by $\text{JI}(L)$ (resp., $\text{MI}(L)$). For $j\in\text{JI}(L)$ (resp., $m\in\text{MI}(L)$), we let $j_{*}$ (resp., $m^{*}$) denote the unique element of $L$ such that $j_{*}\lessdot j$ (resp., $m\lessdot m^{*}$).

For $A\subseteq L$, the expression $\bigvee A := \bigvee_{a \in A} a$ is \textit{irredundant} if there does not exist a proper subset $A^{\pr}\subsetneq A$ such that $\bigvee A^{\pr}=\bigvee A$. Given $A,B\subseteq \text{JI}(L)$ such that $\bigvee A$ and $\bigvee B$ are irredundant and $\bigvee A=\bigvee B$, we set $A\preceq B$ if for each $a\in A$ there exists $b\in B$ with $a\leq b$. In this situation, we say that $\bigvee A$ is a \textit{refinement} of $\bigvee B$. If $x\in L$ and $A\subseteq\text{JI}(L)$ such that $x=\bigvee A$ is irredundant, we say $\bigvee A$ is a \textit{canonical join representation} of $x$ if $A\preceq B$ for any other irrendundant join representation $x=\bigvee B,\ B\subseteq\text{JI}(L)$. Dually, one defines \textit{canonical meet representations}.

Now we assume that $L$ is a \textit{semidistributive} lattice. This means that for any three elements $x, y, z \in L$, the following properties hold:
\begin{itemize}
\item if $x \wedge z = y \wedge z$, then $(x \vee y) \wedge z = x \wedge z$, and
\item if $x \vee z = y \vee z$, then $(x \wedge y) \vee z = x \vee z$.
\end{itemize}
It is known that a lattice $L$ is semidistributive if and only if each element of $L$ has a canonical join representation and a canonical meet representation \cite[Theorem 2.24]{freese1995free}. Let $\Delta^{CJ}(L)$ be the collection of canonical join representations of elements of $L$. There is a canonical bijection $L \to \Delta^{CJ}(L)$ sending $x \mapsto A$ where $\bigvee A$ is the canonical join representation of $x$.


\begin{lemma}\cite[Theorem 1.1]{barnard2016canonical}\label{lem:flag}
  If $L$ is a semidistributive lattice, then $\Delta^{CJ}(L)$ is the set of faces of an abstract simplicial complex, called the canonical join complex. Furthermore, this complex is flag, meaning that $\{j_1,\ldots,j_m\}\subseteq\JI(L)$ is a face if and only if $\{j_a, j_b\}$ is a face for all $a\neq b$.
\end{lemma}


A set map $\lambda : \Cov(L) \to P$, where $(P, \leq_{P})$ is some poset is called an \textit{edge labeling}. 

\begin{definition}\label{defn_cu} An edge labeling  $\lambda: \Cov(L) \to Q$ is a \textit{CN-labeling} if $L$ and its dual $L^{*}$ satisfy the following: given $x,y,z \in L$ with $(z,x), (z,y) \in \Cov(L)$ and maximal chains $C_{1}$ and $C_{2}$ in $[z, x \vee y]$ with $x \in C_{1}$ and $y \in C_{2}$, \begin{enumerate}
\item[(CN$1$)]{the elements $x' \in C_{1}, y' \in C_{2}$ such that $(x', x \vee y), (y', x \vee y) \in \Cov(L)$ satisfy $$\lambda(x', x \vee y) = \lambda(z,y),\ \ \lambda(y', x \vee y) = \lambda(z, x);$$}
\item[(CN$2$)]{if $(u,v) \in \Cov(C_{1})$ with $z < u\lessdot v < x \vee y$, then $\lambda(z,x)<_{Q}\lambda(u,v)$ and $\lambda(z,y) <_{Q} \lambda(u,v)$;}
\item[(CN$3$)]{the labels on $\Cov(C_{1})$ are pairwise distinct.}
\end{enumerate}

We say that $\lambda$ is a \textit{CU-labeling} if, in addition, it satisfies \begin{enumerate}
\item[(CU$1$)]{$\lambda(j_{*}, j) \neq \lambda(j_{*}', j')$ for $j, j' \in \text{JI}(L)$, $j \neq j'$, and}
\item[(CU$2$)]{$\lambda(m, m^{*}) \neq \lambda(m', m'^{*})$ for $m, m' \in \text{MI}(L)$, $m \neq m'$.}
\end{enumerate} 

If $L$ admits a CU-labeling, it is said to be \textit{congruence-uniform}.\end{definition}

\begin{remark}\label{Remark_def_of_CU_lattice}
For completeness, we include the more standard definition of a congruence-uniform lattice. 

Recall that an equivalence relation $\Theta$ on the elements of $L$ is called a \textit{lattice congruence} of $L$ if $\Theta$ satisfies the following:
\begin{itemize}
\item if $x \equiv_\Theta y$, then $x\vee t \equiv_\Theta y \vee t$ and $x\wedge t \equiv_\Theta y \wedge t$ for each $x, y, t \in L$.
\end{itemize}
Let $\text{Con}(L)$ denote the set of all lattice congruences of $L$. The set $\text{Con}(L)$ turns out to be a distributive lattice when its elements are ordered by refinement.

Given $(x, y) \in \text{Cov}(L)$, we let $\text{con}(x,y)$ denote the most refined lattice congruence for which $x \equiv y$. Such congruences are join-irreducible elements of the lattice $\text{Con}(L)$. When $L$ is a finite lattice, the join-irreducibles (resp., meet-irreducibles) of $\text{Con}(L)$ are the congruences of the form $\text{con}(j_*,j)$ (resp., $\text{con}(m,m^*)$). We thus obtain surjections 
$$\begin{array}{rclccrcl}
\text{JI}(L) & \rightarrow & \text{JI}(\text{Con}(L)) & & &  \text{MI}(L) & \rightarrow & \text{MI}(\text{Con}(L)) \\
j & \mapsto & \text{con}(j_*,j) & & &  m & \mapsto & \text{con}(m,m^*).
\end{array}$$
If these maps are bijections, we say that $L$ is \textit{congruence-uniform}. It is known that this definition and the one given in Definition~\ref{defn_cu} are equivalent (for instance, see \cite[Proposition 2.5]{garver2016oriented}).
\end{remark}

We conclude this section by mentioning some general properties of CU-labelings and the definition of the lattice-theoretic shard intersection order of $L$. Given an edge labeling $\lambda: \text{Cov}(L) \to P$, one defines \[ \lambda_{\downarrow}(x) := \{\lambda(y,x): \ y\lessdot x\}, \ \ \ \lambda^{\uparrow}(x):= \{\lambda(x,z): \ x \lessdot z\}. \]

\begin{lemma}\cite[Lemma 2.6]{garver2016oriented}\label{cu-label_ji_mi}
Let $L$ be a congruence-uniform lattice with CU-labeling $\lambda:\Cov(L)\ra P$. For any $s\in \lambda(\Cov(L))$, there is a unique join-irreducible $j \in \text{JI}(L)$ (resp., meet-irreducible $m \in \text{MI}(L)$) such that $\lambda(j_*,j) = s$ (resp., $\lambda(m,m^*) = s$). Moreover, this join-irreducible $j$ (resp., meet-irreducible $m$) is the minimal (resp., maximal) element of $L$ such that $s \in \lambda_\downarrow(j)$ (resp., $s \in \lambda^\uparrow(m)$).
\end{lemma}

We will use Lemma~\ref{cu-label_ji_mi} to characterize join- and meet-irreducible elements of $\text{Bic}(A)$, the lattice of biclosed sets of strings defined in Section~\ref{sec:biclosed}.

One also uses CU-labelings to determine canonical join representations and canonical meet representations of elements of a congruence-uniform lattice. We state this precisely as follows.

\begin{lemma}\cite[Proposition 2.9]{garver2016oriented}\label{Lemma_cjr_cmr}
Let $L$ be a congruence-uniform lattice with CU-labeling $\lambda$. For any $x\in L$, the canonical join representation of $x$ is $\bigvee D$, where $D = \{j \in \text{JI}(L): \ \lambda(j_*,j) \in \lambda_\downarrow(x)\}$. Dually, for any $x\in L$, the canonical meet representation of $x$ is $\bigwedge U$, where $U = \{m \in \text{MI}(L): \ \lambda(m,m^*) \in \lambda^\uparrow(x)\}$.
\end{lemma}

\begin{definition}
Let $L$ be a finite congruence-uniform lattice with CU-labeling $\lambda: \Cov(L) \to P.$ Let $x \in L$, and let $\lambda_\downarrow(x) = \{y_1, \ldots, y_k\}$. Define the \textit{shard intersection order} of $L$, denoted $\Psi(L)$, to be the collection of sets of the form $$\psi(x) := \{\lambda(w,z) | \wedge_{i = 1}^k y_i \le w \lessdot z \le x\}$$ partially ordered by inclusion. 
\end{definition}

\section{Representation theory preliminaries}\label{sec:representation_prelim}

\subsection*{Notations and Conventions}

Throughout, $k$ denotes a field, $\Lambda$ a $k$-algebra, and $\modu(\Lambda)$ the category of all finitely generated left $\Lambda$-modules. 
For a subcategory $\mathcal{C}$ of $\modu(\Lambda)$, we always assume $\mathcal{C}$ is full and closed under isomorphism.
We let $\ind(\Lambda)$ denote the set of all isomorphism classes of indecomposable modules in $\modu(\Lambda)$. For every $M$ in $\modu(\Lambda)$, we denote the Auslander-Reiten translation of $M$ by $\tau_{\Lambda} M$.

A \textit{quiver} $Q=(Q_0,Q_1,s,t)$ is a directed graph, which consists of two sets $Q_0$ and $Q_1$ and two functions $s, t:Q_1\to Q_0$. Elements of $Q_0$ and $Q_1$ are called \textit{vertices} and \textit{arrows} of $Q$, respectively. For $\gamma \in Q_1$, the vertex $s(\gamma)$ is its \textit{source} and $t(\gamma)$ is its \textit{target}. We will assume that $Q$ is finite and connected. We typically use lower case Greek letters $\alpha$, $\beta$, $\gamma$, $\ldots$ for arrows of $Q$. 

A \emph{path of length} $d\geq 1$ in $Q$ is a finite sequence of arrows $\gamma_{d}\cdots\gamma_{2}\gamma_{1}$ such that $s(\gamma_{j+1})=t(\gamma_{j})$, for every $1 \leq j \leq d-1$. We also associate to each vertex $i \in Q_0$ a path of length $0$, denoted $e_i$, called the \textit{lazy path}. Each lazy path $e_i$ satisfies $s(e_i) = t(e_i) = i.$ The \emph{path algebra} of $Q$, denoted $kQ$, is generated by the set of all such paths and all of the lazy paths as a $k$-vector space. Its multiplication is induced by concatenation of paths and extended to $kQ$ by linearity. Let $R_Q \subseteq kQ$ denote the two-sided ideal generated by all arrows of $Q$. A two-sided ideal $I \subseteq kQ$ is called \emph{admissible} if $R_Q ^m \subseteq I \subseteq R_Q^2$, for some $m\geq 2$.

More details on the representation theory of associative algebras that appears in this paper may be found in \cite{ASS}.

\subsection{Gentle Algebras}\label{String Algebra}
In this subsection, we recall some basic notions about gentle algebras, which are used in the remainder of the paper. For further details we refer the reader to \cite{BR}. 

A finite dimensional algebra $\Lambda = kQ/I$, where $I$ is an admissible ideal generated by a set of paths, is called a \emph{string algebra} if the following conditions hold:

\begin{enumerate}[(S1)]
\item At every vertex $v \in Q_0$, there are at most two incoming and two outgoing arrows.
\item For every arrow $\alpha \in Q_1$, there is at most one arrow $\beta$ and one arrow $\gamma$ such that $\alpha\beta \notin I$ and $\gamma\alpha \notin I$.
\end{enumerate}

Moreover, $\Lambda = kQ/I$ is called \emph{gentle}, if it also satisfies the following:

\begin{enumerate}[(G1)]
\item There is a set of paths of length two that generate $I$.
\item For each arrow $\alpha \in Q_1$, there is at most one $\beta$ and one $\gamma$ such that $0\ne\alpha\beta \in I$ and $0\ne\gamma\alpha \in I$.
\end{enumerate}

Unless otherwise stated, given a finite dimensional algebra $\Lambda = kQ/I$, we assume that $I$ is an admissible ideal generated by a set of paths.

\subsection*{Strings and Band Modules}
Let $\Lambda=kQ/I$ and $Q_1^{-1}$ be the set of formal inverses of arrows of $Q$. Elements of $Q_1^{-1}$ are denoted by $\gamma^{-1}$, where $\gamma \in Q_1$, such that $s(\gamma^{-1}):=t(\gamma)$ and $t(\gamma^{-1}):=s(\gamma)$.
A \emph{string} in $\Lambda$ of length $d\geq 1$ is a word $w = \gamma_d^{\epsilon_d}\cdots\gamma_1^{\epsilon_1}$ in the alphabet $Q_1 \sqcup Q_1^{-1}$ with $\epsilon_i \in \{\pm 1\}$, for all $i \in \{1,2,\cdots,d\}$, which satisfies the following conditions:

\begin{enumerate}
\item[(P1)] $s(\gamma_{i+1}^{\epsilon_{i+1}})=t(\gamma_i^{\epsilon_i})$ and $ \gamma_{i+1}^{\epsilon_{i+1}} \neq \gamma_i^{-\epsilon_i}$, for all $i\in \{1,\cdots,d-1 \}$;
\item[(P2)] $w$ and also $w^{-1} := \gamma_1^{-\epsilon_1}\cdots\gamma_d^{-\epsilon_d}$ do not contain a subpath in $I$.
\end{enumerate}

If $w = \gamma_d^{\epsilon_d}\cdots\gamma_1^{\epsilon_1}$ and we know that $\epsilon_i = 1$ for some $i$, we write $\gamma_i$ rather than $\gamma_i^{1}$. We say $w$ \textit{starts} at $s(w)=s(\gamma_1^{\epsilon_1})$ and \textit{terminates} at $t(w)=t(\gamma_d^{\epsilon_d})$. We also associate a zero-length string to every vertex $i \in Q_0.$ We denote this string by $e_i$. We let $\Str(\Lambda)$ denote the set of strings in $\Lambda$ where a string $w$ is identified with $w^{-1}$ for reasons that will become clear later. 


Let $w = \gamma_d^{\epsilon_d}\cdots\gamma_1^{\epsilon_1}$ be in $\Str(\Lambda)$. Then, $w$ is called \emph{direct} if $\epsilon_i=1$ for all $i\in \{1,\cdots,d \}$, while inverse strings are defined dually. We say that a string $w$ of positive length is a \emph{cyclic} string if $s(w) = t(w)$. If $w$ is a cyclic string, it is called a \emph{band} if $w^m$ is a string for each $m \in \mathbb{Z}_{\geq 1}$ and $w$ is not a power of string of a strictly smaller length.

Let $w =  \gamma_d^{\epsilon_d}\cdots \gamma_1^{\epsilon_1}$ be an element of $\Str(\Lambda)$. We can express the walk on $Q$ determined by the string $w$ as the sequence \xymatrix{x_{d+1} \ar@{-}^{\gamma_d}[r] & x_d \ar@{-}^{\gamma_{d-1}}[r] & \cdots & x_1 \ar@{-}_{\gamma_1}[l]} where $x_1,\ldots, x_{d+1}$ are the vertices of $Q$ visited by $w$, \emph{a priori} multiple times. Each arrow $\gamma_i$ has an orientation that we suppress in this notation, but the orientation of these arrows appears in the definition of the string module defined by $w$. The \textit{string module} defined by $w$ is the quiver representation $ M(w) := ((V_i)_{i \in Q_0}, (\varphi_\alpha)_{\alpha\in Q_1})$ with vector spaces given by
$$\begin{array}{cccccccccccc}
V_i & := & \left\{\begin{array}{ccl} \displaystyle\bigoplus_{j: x_j = i}kx_j &: & \text{if } i = x_j \text{ for some } j \in \{1,\ldots, d+1\}\\ 0 & : & \text{otherwise} \end{array}\right.
\end{array}$$ for each $i \in Q_0$ and with linear transformations given by
$$\begin{array}{cccccccccccc}
\varphi_\alpha(x_k) & := & \left\{\begin{array}{lcl} x_{k-1} &: & \text{if } \alpha = \gamma_{k-1} \text{ and } \epsilon_k = -1\\ x_{k+1} &: & \text{if } \alpha = \gamma_{k} \text{ and } \epsilon_k = 1\\  0 & : & \text{otherwise} \end{array}\right.
\end{array}$$ for each $\alpha \in Q_0$. Observe that $\dim_k(V_i) = |\{j \in \{1,\ldots, d+1\}: \ x_j = i\}|$ for any $i \in Q_0$. Observe that for any string $w$ we have that $M(w) \cong M(w^{-1})$ as $\Lambda$-modules.

As shown in \cite{wald1985tame}, all of the indecomposable modules over a string algebra are given by string modules and another class called \emph{band modules}. As band modules will not be relevant in this work, we do not define them, instead we refer the interested reader to \cite{BR}. 

The \emph{diagram} of $w$ is a pictorial presentation of $M(w)$ that consists of a sequence of up and down arrows, drawn from right to left. In particular, starting from vertex $s(w)$, for every direct arrow we put a left-down arrow outgoing from the current vertex, whereas for each inverse arrow we put a right-down ending at the current vertex. These notions, as well as the construction of a string module, are illustrated in the following example.

\begin{example}
Let $(Q,I)$ denote the bound quiver where $Q$ appears in Figure~\ref{fig_gentle_quiver} and $I = \langle\beta\alpha \rangle$. Since $R_Q^4 = 0$, the zero ideal is admissible. Furthermore, the ideal $I$ generated by the quadratic relation $\beta \alpha$ is also an admissible ideal, for which the quotient algebra $\Lambda=kQ/I$ is gentle.

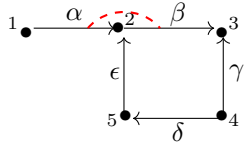
\begin{figure}[!htbp]
\begin{center}
\begin{tikzpicture}[scale=0.6]
\draw [->] (-2.2,0) --(-0.4,0);
\node at (-1.3,0.3) {$\alpha$};
\node at (-2.5,0) {$^1\bullet$};
\draw [->] (0.2,0) --(1.8,0);
\node at (1,0.3) {$\beta$};
\node at (-0.2,0.1) {$\bullet^2$};
\draw [->] (-0.2,-1.9) -- (-0.2,-0.2);
\node at (-0.3,-2) {$_5\bullet$};
\node at (-0.4,-1) {$\epsilon$};
\draw [->] (2,-2) --(2,-0.2) ;
\node at (2.1,-2) {$\bullet_4$};
\node at (2.3,-1) {$\gamma$};
\draw [->] (1.9,-2) --(0,-2);
\node at (1,-2.3) {$\delta$};
\node at (2.1,0) {$\bullet^3$};
-------
\draw[] (-.2,0)--(0.2,0);
  \draw [thick, dashed] [color=red] (-1,0) to [bend left=50] (0.6,0);
\end{tikzpicture}
\end{center}
\caption{A bound quiver $(Q,I)$ where $\Lambda = kQ/I$ is gentle. The red dashed arc depicts the principal generator of $I$.}
\label{fig_gentle_quiver}
\end{figure}


Observe that $w =\alpha^{-1}\epsilon \delta \gamma^{-1}\beta$ is a string in $\Str(\Lambda)$. The diagram of $w$ and the string module $M(w)$ appear in Figure~\ref{fig_diag_and_module}.

\begin{figure}[!htbp]
$$\begin{array}{ccccccc}
\begin{tikzpicture}[scale=0.67]


-------------------------
-------------------------
\node at (0.1,0.1) {$\bullet^2$};
\draw [->] (0,0) -- (-1,-1);
\node at (-0.25,-0.5) {$\beta$};
\node at (-0.9,-1.1) {$\bullet_3$};
--
\draw [->] (-2,0) -- (-1.05,-1);
\node at (-1.25,-0.2) {$\gamma^{-1}$};
\node at (-2.15,0.15) {$^4\bullet$};
--
\draw [->] (-2.1,0) -- (-3,-1);
\node at (-2.25,-0.5) {$\delta$};
\node at (-2.9,-1.1) {$\bullet_5$};
--
\draw [->] (-3.1,-1.07) -- (-4,-2.05);
\node at (-3.25,-1.5) {$\epsilon$};
\node at (-3.9,-2.12) {$\bullet_2$};
--
\draw [->] (-5,-1) -- (-4.05,-2);
\node at (-4.25,-1.2) {$\alpha^{-1}$};
\node at (-5.1,-.95) {$^1\bullet$};
---------------------------------
\end{tikzpicture} & &
\begin{tikzpicture}[scale=0.67]
\draw [->] (5.8,0) --(7.6,0);
\node at (6.6,-.7) {$\begin{bmatrix}
    0 \\
    1 \\
    \end{bmatrix}$};
\node at (5.5,-0.1) {$\bullet$};
\node at (5.5,0.3) {$k$};
\draw [->] (7.8,0) -- (9.8,0);
\node at (9,0.4) {$\begin{bmatrix}
    1 & 0\\
    \end{bmatrix}$};
\node at (7.7,-0.1) {$\bullet$};
\node at (7.9,0.4) {$k^2$};
\draw [->] (7.7,-1.9) -- (7.7,-0.2);
\node at (7.7,-2.1) {$\bullet$};
\node at (7.7,-2.5) {$k$};
\node at (7.3,-1.2) {$\begin{bmatrix}
    0 \\
    1 \\
    \end{bmatrix}$};
\draw [->] (10,-2) --(10,-0.2) ;
\node at (10,-2.1) {$\bullet$};
\node at (10.1,-2.5) {$k$};
\node at (10.3,-1) {$1$};
\draw [->] (9.9,-2.1) --(8,-2.1);
\node at (9,-2.4) {$1$};
\node at (10,0) {$\bullet$};
\node at (10.1,0.4) {$k$};

\end{tikzpicture} \\ (a) & & (b)
\end{array}$$
\caption{In $(a)$, we show the diagram of $w =\alpha^{-1}\epsilon \delta \gamma^{-1}\beta$, and, in $(b)$, we show the string module $M(w)$.}
\label{fig_diag_and_module}
\end{figure}
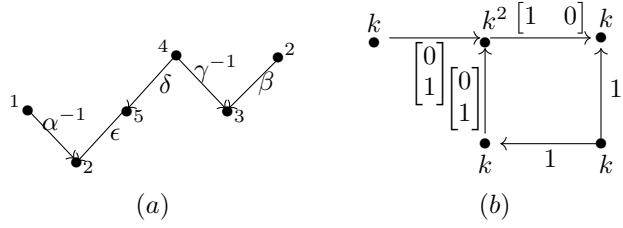

\end{example}

\subsection{Torsion theories}
Following the seminal work of Dickson \cite{Di}, a subcategory of $\modu(\Lambda)$ is called a \emph{torsion class} if it is closed under quotients and extensions. We say a torsion class $\mathcal{T}$ is {\it functorially finite} if $\mathcal{T}=\Gen (M)$, for some $\Lambda$-module $M$, where $\Gen (M)$ denotes the subcategory of $\modu(\Lambda)$ \emph{generated} by $M$ (i.e., the subcategory consisting of all quotients of direct sums of $M$).

Dually, a {\emph{torsion-free class}} is defined as a subcategory of $\modu(\Lambda)$ that is closed under submodules and extensions. 
Furthermore, for a subcategory $\mathcal{C}$ of $\modu(\Lambda)$, if we define
$$\mathcal{C}^\bot:= \{X\in \modu(\Lambda) \,| \, \Hom_\Lambda(C,X)=0, \, \forall C\in \mathcal{C}\},$$
then it is easy check that $\mathcal{F}:= \mathcal{T}^{\bot}$ is a torsion free class, provided that $\mathcal{T}$ is a torsion class. In such a case, $(\mathcal{T}, \mathcal{F})$ is called a \emph{torsion pair} or \emph{torsion theory} in $\modu(\Lambda)$.

Let $\tors(\Lambda)$ and $\torf(\Lambda)$, respectively, denote the set of all torsion classes and torsion free classes in $\modu(\Lambda)$, ordered by inclusion. It is straightforward to show these are complete lattices where the meet of a family of torsion classes $\{\mathcal{T}_i\}_{i \in I} \in \tors(\Lambda)$ (resp., $\{\mathcal{F}_i\}_{i \in I} \in \torf(\Lambda)$) is given by $\bigwedge_{i\in I}\mathcal{T}_i = \bigcap_{i \in I}\mathcal{T}_i$ (resp., $\bigwedge_{i\in I}\mathcal{F}_i = \bigcap_{i \in I}\mathcal{F}_i$). Moreover, these lattices are closely related via an anti-isomorphism of lattices by sending $\mathcal{T}$ to $\mathcal{T}^{\bot}$ (and $\mathcal{F}$ to $^{\bot}\mathcal{F}$ in the opposite direction), where $$ ^{\bot}\mathcal{C} := \{X \in \modu(\Lambda) \,| \, \Hom_\Lambda(X,C) = 0, \, \forall C \in \mathcal{C}\}$$ for each subcategory $\mathcal{C}$ of $\modu(\Lambda)$.

The following proposition will be useful in the following sections, as it describes the smallest torsion class in $\modu(\Lambda)$ containing a given set of modules. Later we use a refinement of this proposition for a combinatorial description of torsion classes over gentle algebras. Recall that for a subcategory $\mathcal{C}$ of $\modu(\Lambda)$, the smallest extension-closed subcategory of $\modu(\Lambda)$ that contains $\mathcal{C}$ consists of all modules in $\modu(\Lambda)$ which have a filtration by the objects in $\mathcal{C}$. We denote this category by $\Filt(\mathcal{C})$.

\begin{proposition}\label{Smallest Torsion Class}

For a collection of $\Lambda$-modules $X_1,\ldots, X_r$, the smallest torsion class in $\tors(\Lambda)$ that contains $\{X_1, \ldots, X_r \}$ is given by $\mathcal{T}^*=\Filt(\Gen(\bigoplus_{i = 1}^rX_i))$. In particular, each $M$ in $\mathcal{T}^*$ has a filtration $0= M_0 \subseteq M_1 \subseteq \cdots \subseteq M_{d-1} \subseteq M_d=M$ such that for every $1 \leq i \leq d$, there exists an epimorphism $\psi_i:X_{j_i}\twoheadrightarrow M_{i}/M_{i-1}$ for some $1 \leq j_i \leq r$.
\end{proposition}
\begin{proof}
We prove that $\mathcal{T}^*$ is a torsion class and is contained in any $\mathcal{T} \in \tors(\Lambda)$ which contains the modules $X_1, \ldots, X_r$.

To show the containment, suppose $\mathcal{T} \in \tors(\Lambda)$ and $X_1,\ldots, X_r$ belong to $\mathcal{T}$.
If $M \in \mathcal{T}^*$, by definition it has a filtration 
$$ 0= M_0 \subseteq M_{1} \subseteq \cdots \subseteq M_{d-1} \subseteq M_d=M$$ 
such that for each $1 \leq i \leq d$, there exists an epimorphism $\psi_i:X_{j_i}\twoheadrightarrow M_i/M_{i-1}$ for some $1\leq j_i \leq r$. Note that $M_{1} \in \mathcal{T}$. Now, via an inductive argument and the fact that $\mathcal{T}$ is extension-closed, for any $1\leq i \leq d$, the short exact sequence 
$$ 0 \rightarrow M_{i-1} \rightarrow M_{i} \rightarrow M_{i}/M_{i-1} \rightarrow 0$$
implies $M_i \in \mathcal{T}$. In particular, $M \in \mathcal{T}$ so $\mathcal{T}^* \subseteq \mathcal{T}$.

To show that $\mathcal{T}^*$ is a torsion class, consider $M, N \in \mathcal{T}^*$, respectively, with the following filtrations
$$ 0= M_0 \subseteq M_{1} \subseteq \cdots \subseteq M_{a-1} \subseteq M_a=M$$
and 
$$ 0= N_0 \subseteq N_{1} \subseteq \cdots \subseteq N_{b-1} \subseteq N_b=N$$ 
such that the module epimorphisms $\alpha_i:X_{j_i}\twoheadrightarrow M_i/M_{i-1}$ and $\beta_{i^\prime}:X_{k_t}\twoheadrightarrow N_{i^\prime}/N_{i^{\prime}-1}$ are defined as before, for every $ 1\leq i \leq a$ and $1\leq i^\prime \leq b$.

Suppose we have the following short exact sequence in $\modu(\Lambda)$:
$$ 0 \rightarrow N\xrightarrow{f} Z \xrightarrow{g} M \rightarrow 0.$$
Consider the following filtration of $Z$ with the desired quotient property:
$$0=f(N_0) \subseteq \cdots \subseteq f(N_b)=g^{-1}(M_0) \subseteq \cdots \subseteq g^{-1}(M_a)=Z.$$ 

Using the maps $\alpha_i$ and $\beta_{i^\prime}$ given above, it is straightforward to show that each quotient of two consecutive terms in this filtration of $Z$ is a quotient of $X_k$ for some $1 \le k \le r$. This proves that $\mathcal{T}^*$ is extension-closed.

To see that $\mathcal{T}^*$ is quotient closed, suppose $f: M \twoheadrightarrow N $ is an epimorphism and 
$$0= M_0 \subseteq M_{1} \subseteq \cdots \subseteq M_{d-1} \subseteq M_d=M$$ 
a filtration of $M$ as in the assertion. Now consider the filtration 
$$0= f(M_0) \subseteq f(M_{1}) \subseteq \cdots \subseteq f(M_{d-1}) \subseteq f(M_d)=N$$ 
in which some of the middle terms might be the same. Each map $\psi_i$ from the original filtration gives rise to an epimorphism $\sigma_i:X_{j_i} \twoheadrightarrow f(M_i)/f(M_{i-1})$. Therefore $\mathcal{T}^*$ is a torsion class of $\modu(\Lambda),$ and we are done.
\end{proof}

\section{Brick gentle algebras}\label{sec:gentle_alg_prelim}

Recall that a module $X$ over a $k$-algebra $\Lambda$ is called a \emph{brick} if $\End_{\Lambda}(X)$ is a division ring. We say that $\Lambda$ is a \emph{brick algebra} if every indecomposable $\Lambda$-module is a brick. It is well-known that $X$ is a brick if and only if $\End_{\Lambda}(X)\simeq k$, provided that $k$ is algebraically closed. It follows from \cite[Remark, Lemma 4 in Section 3]{bongartz1991geometric} that any brick algebra is of finite representation type.


In this section, we classify the gentle algebras that are brick algebras. For the remainder of the paper, we will refer to such algebras as \emph{brick gentle algebras}. We show that all strings in such bound quivers are \emph{self-avoiding}, meaning that no string revisits a vertex. In particular, over brick gentle algebras, the sets of string modules, bricks, and indecomposable $\tau$-rigid modules coincide. 


Recall that a $\Lambda$-module $M$ is called \emph{rigid} (resp., $\tau$-rigid) if $\Ext^1_{\Lambda}(M,M)=0$ (resp., $\Hom(M,\tau M) = 0$). Here $\tau$ denotes the \textit{Auslander-Reiten translation}. From the functorial isomorphism
$\Ext^1_{\Lambda}(Y,X) \simeq D \overline{\Hom}_{\Lambda}(X,\tau_{\Lambda} Y) $, known as \emph{Auslander-Reiten duality}, it is follows that every $\tau$-rigid module is rigid. 

To avoid repetition, we fix some notation that will be used throughout this section. Let $A$ denote a gentle algebra with fixed bound quiver $(Q,I)$. Let $w= \gamma^{\epsilon_d}_d\cdots\gamma^{\epsilon_2}_2\gamma^{\epsilon_1}_1$ be in $\Str(A)$, with $\gamma_i \in Q_1$ and $\epsilon_i \in \{\pm 1\}$, for every $1 \leq i \leq d$. 
We say that $\gamma_j\gamma_i$ and $\gamma_i^{-1}\gamma_j^{-1}$ is a \emph{relation}, if $\gamma_j\gamma_i$ is a path of length two in $Q$ which belongs to $I$. By $\Brick(A)$ and $\tau \text{-rigid} (A)$, we respectively denote the set of bricks and $\tau_A$-rigid modules in $\operatorname{mod}(A)$. 

The next lemma gives a simple criterion for showing that the string modules defined by certain cyclic strings are not bricks. In particular, it shows that if a bound quiver of an algebra $\Lambda$ contains a cyclic string of odd length, the set $\ind(\Lambda)\backslash\Brick(\Lambda)$ is nonempty.


\begin{lemma}\label{Non-brick}
Let $w= \gamma^{\epsilon_d}_d\cdots\gamma^{\epsilon_2}_2\gamma^{\epsilon_1}_1$ be a cyclic string so that $s(w)=t(w)$. If there exists $1\leq i \leq d-1$ such that $\epsilon_{i}=\epsilon_{i+1}$ or $\epsilon_{d}=\epsilon_{1}$, then $M(w)$ is not a brick. In particular, if $\gamma_1 \gamma_d$ is a relation, $w$ is not a brick.
\begin{proof}
Assume that $\epsilon_{d}=\epsilon_{1}$, and let $j := t(w)=s(w)$. Consider $f\in \End_A(M(w))$, given by $f=f_2 \circ f_1$, where $f_1:M(w) \rightarrow M(e_j)$ (resp., $f_2:M(e_j) \rightarrow M(w)$) is the surjection onto (resp., injection from) the simple module $M(e_j)$. Obviously, $f$ is nonzero and not invertible, which implies that $M(w)$ is not a brick. The proof for the other case is analogous.
\end{proof}
\end{lemma}


We define a \emph{walk} in a quiver $Q$ of length $d \ge 1$ to be a word $w = \gamma_d^{\epsilon_d}\cdots\gamma_1^{\epsilon_1}$ in the alphabet $Q_1 \sqcup Q_1^{-1}$ with $\epsilon_i \in \{\pm 1\}$, for all $i \in \{1,2,\cdots,d\}$, and which satisfies condition (P1) in the definition of a string in $A$. When working with walks in a quiver, we use analogous notation to that which is used for strings.

\begin{proposition}\label{Ind-Brick Finite}
For a gentle algebra $A=kQ/I$, the following are equivalent:

\begin{enumerate}
\item $A$ is a brick algebra;
\item Every cyclic walk $w = \gamma_d^{\epsilon_d}\cdots\gamma_1^{\epsilon_1}$ in $Q$ contains at least two relations.
\end{enumerate}
Therefore any string $w \in \Str(A)$ where $A$ is any brick gentle algebra is self-avoiding.
\end{proposition}

\begin{proof}
If there exists a cyclic walk $w$ in the bound quiver $(Q,I)$ that contains no relations, then there exists a band in $A$. This contradicts that $A$ is representation finite. If there exists a cyclic walk $w$ in the bound quiver $(Q,I)$ that contains a single relation, then by Lemma~\ref{Non-brick} this contradicts that $\ind(A) \subseteq \Brick(A)$. We obtain that $(2)$ is a consequence of $(1)$.


Conversely, $(2)$ implies that each string $w \in \Str(A)$ never revisits a vertex. Thus $\End(M(w)) \simeq k$.
\end{proof}


\section{Biclosed sets and biclosed subcategories}\label{sec:biclosed}

In this section, we recall the definition of the lattices of biclosed sets and biclosed subcategories, we construct a CU-labeling for these lattices, and we classify the join-irreducible biclosed sets.

A subcategory $\mathcal{C}$ of $\modu(\Lambda)$ is called \emph{weakly extension-closed} provided that for every triple of indecomposables $X,Y$ and $Z$ in $\modu(\Lambda)$ in a short exact sequence $0 \rightarrow X \rightarrow Y \rightarrow Z \rightarrow 0$, if $X$ and $Z$ in $\mathcal{C}$, then $Y \in \mathcal{C}$. Moreover, $\mathcal{C}$ is \emph{biclosed} if both $\mathcal{C}$ and $\mathcal{C}^c$ are weakly extension-closed, where $\mathcal{C}^c:= \{X\in \modu(A) \,| \, \add(X) \cap \mathcal{C}=0\}$.

%
%
%
%

In \cite{GM}, the first and second authors studied the poset of biclosed sets of strings which is the combinatorial incarnation of biclosed subcategories. Before defining this poset, we define a \emph{concatenation} of two strings $u, v \in \Str(\Lambda)$ to be a string in $\Str(\Lambda)$ of the form $v\gamma u$ or $v\gamma^{-1} u$, provided there exists such an arrow $\gamma \in Q_1$. At times, we will denote a concatenation of two strings $u$ and $v$ by $v\gamma^{\pm 1} u$ when we do not wish to specify whether we are considering $v\gamma u$ or $v\gamma^{-1} u$.

Now, a subset $B$ of $\Str(A)$ is called \emph{closed} if $u, v \in B$ implies that $v\gamma^{\pm 1} u$ is also in $B$, provided that $v\gamma^{\pm 1} u \in \Str(A)$ for some $\gamma \in Q_1$. Moreover, $B$ is called \emph{biclosed} if $B$ and $B^c:=\Str(A) \backslash B$ are closed. 
In order to distinguish the combinatorially defined biclosed sets from the homologically defined biclosed subcategories, we respectively denote these by $\Bic(A)$ and $\bic(A)$. Subsequently, $B$ and $\mathcal{B}$, respectively, will denote a biclosed set and a biclosed subcategory.

Both sets $\Bic(A)$ and $\bic(A)$ are partially ordered by inclusion. We leave it to the reader to verify that the map
$$B \mapsto\add \bigg( \bigoplus M (w) |\, w \in B \bigg)$$ 
defines a poset isomorphism between $\Bic(A)$ and $\bic(A)$.

\begin{example}\label{Ex_2cyc_alg}
Consider the brick gentle algebra $$A = k(\xymatrix{1 \ar@<1ex>[r]^\alpha & 2 \ar@<1ex>[l]^\beta})/\langle \alpha\beta, \beta\alpha\rangle.$$ In Figure~\ref{fig_A2_preproj}, we show the poset $\Bic(A)$. 
\end{example}

\begin{figure}[!htbp]
$$\includegraphics[scale=1]{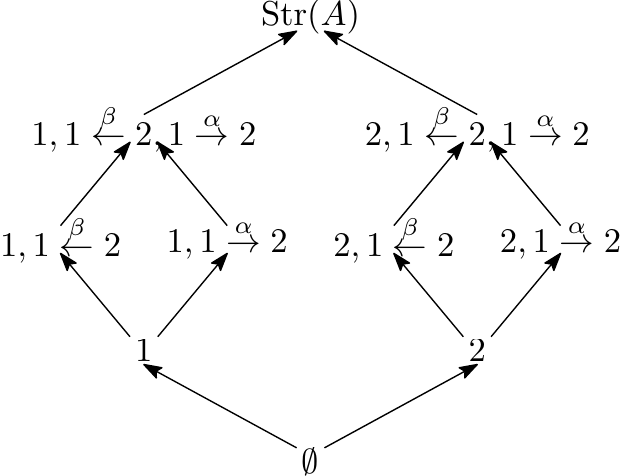}$$
\caption{A poset of biclosed sets of strings.}
\label{fig_A2_preproj}
\end{figure}

The following lemma describes the lattice structure of $\Bic(A)$ are $\bic(A)$.

\begin{theorem}\cite[Theorem 3.26]{palu2017non}\label{thm_bic_is_cu}
If $A$ is a representation finite gentle algebra, the poset $\Bic(A)$ is a congruence-uniform lattice.
\end{theorem}

For the remainder of the section, we assume that $A$ is a brick gentle algebra. It follows from \cite[Theorem 3.20 (ii)]{palu2017non} that for any biclosed sets $B_1, B_2 \in \Bic(A)$, one has that $B_1\vee B_2 = \overline{B_1\cup B_2}$ where for any $X \subseteq \Str(A)$ the set $\overline{X}$ denotes the smallest closed subset of $\Str(A)$ that contains $X$. The proof of \cite[Theorem 3.26]{palu2017non} shows that if $B\sqcup \{v\}, B \sqcup \{w\} \in \Bic(A)$, then $$(B\sqcup \{v\})\vee (B \sqcup \{w\}) = B\sqcup \overline{\{v,w\}}.$$

We now construct a CU-labeling for the lattice $\Bic(A)$. Let $w = \gamma_d^{\epsilon_d}\cdots \gamma_1^{\epsilon_1} \in \Str(A)$ and let $j \in \{1, \ldots, d\}$. We say that a pair $\{w_1, w_2\}$ is a \textit{break} of $w$ if $w = w_1\gamma_j^{\pm 1}w_2$ for some $j \in \{1, \ldots, d\}$. We refer to the strings $w_1$ and $w_2$ in a break of $w$ as \textit{splits} of $w$.

Define a poset $\mathcal{S}$ whose elements are of the form $(w, \{w^1,\ldots, w^{d}\}) \in \text{Str}(A)\times 2^{\text{Str}(A)}$ where 
\begin{itemize}
\item each $w^i$ is a split of $w$, and
\item two distinct splits $w^i$ and $w^j$ do not appear in the same break of $w$
\end{itemize}
up to the equivalence relation where we say that $(w,\{w^1,\ldots, w^d\})$ is equivalent to $(w^{-1}, \{(w^1)^{-1}, \ldots, (w^d)^{-1}\})$. We refer to elements of $\mathcal{S}$ as \textit{labels}, and, for brevity, we denote $(w, \{w^1,\ldots, w^{d}\}) \in \mathcal{S}$ by $w_\mathcal{D}$ with $\mathcal{D} = \{w^1,\ldots, w^{d}\}$. 

We now define the partial order on elements of $\mathcal{S}$. If $u, w \in \Str(A)$, we say that $u$ is a \textit{proper substring} of $w$ if there exist $u^1, u^2 \in \Str(A)$ at most one of which is the empty string such that $w = u^1\gamma_1^{\pm 1}u \gamma_2^{\pm 1}u^2$ for some arrows $\gamma_1, \gamma_2 \in Q_1$. If exactly one of $u^1$ and $u^2$ exists, then only one of the arrows $\gamma_1$ and $\gamma_2$ necessarily exists. The partial order is as follows: given $w_{\{w^1,\ldots, w^{d}\}}, u_{\{u^1,\ldots, u^{e}\}} \in \mathcal{S}$, we say $u_{\{u^1,\ldots, u^{e}\}} \le_\mathcal{S} w_{\{w^1,\ldots, w^{d}\}}$ if $u$ is a proper substring of $w$ or $u_{\{u^1,\ldots, u^{e}\}}$ is equivalent to $w_{\{w^1,\ldots, w^{d}\}}$.

\begin{remark}
A version of this poset of labels $\mathcal{S}$ has already been introduced in \cite{cliftondillerygarver}. There the notion of segments plays the role of strings. Many of the proofs \cite{cliftondillerygarver} are applicable to the current work, and so we will frequently cite \cite{cliftondillerygarver} in the sequel. We leave it to the reader to translate the relevant statements in terms of segments from \cite{cliftondillerygarver} into statements in terms of strings in the current work.
\end{remark}

In the proof of Theorem~\ref{thm_bic_is_cu}, it is shown that any covering relation in the lattice of biclosed sets is of the form $(B, B\sqcup \overline{\{w\}}) \in \text{Cov}(\Bic(A))$ where $w \not \in B$ is a string such that $B$ contains exactly one split from each break of $w$. The following lemma shows that any cover of a biclosed set $B$ is obtained by adding a single string to $B$. 

\begin{lemma}\label{Lemma_covers_in_bic}
For any string $w \in \Str(A)$, we have that $\overline{\{w\}} = \{w\}$. Thus, any covering relation in $\Bic(A)$ is of the form $(B, B\sqcup \{w\})$ where $w \not \in B$ is a string such that $B$ contains exactly one split from each break of $w$.
\end{lemma}

\begin{proof}
Recall that there is a bijection between bricks of $A$ and indecomposable $\tau$-rigid modules of $A$ by \cite[Theorem 1.11]{DIJ}. Using this and that $A$ is a brick gentle algebra, we obtain that every indecomposable $A$-module is $\tau$-rigid. In particular, every indecomposable $A$-module is rigid. This implies that the expression $w\alpha^{\pm 1}w$ is not a string and no string in $A$ may contain this expression. Consequently, $\overline{\{w\}} = \{w\}$.

The second assertion follows from the first.
\end{proof}

\begin{definition}
Define a map ${\lambda}:\text{Cov}(\text{Bic}(A))\rightarrow \mathcal{S}$ by
 ${\lambda}(B,B\sqcup{\{{w}\}})=w_{\{w^1,\ldots,w^{d}\}}$ where $w^1, \ldots, w^{d}$ are the splits of $w$ which are contained in $B$. It is clear that ${\lambda}$ is an edge-labeling of $\text{Bic}(A)$. 
\end{definition}

\begin{proposition}\label{Prop_CU_labeling_bic}
The edge-labeling ${\lambda}: \Cov(\Bic(A))\rightarrow \mathcal{S}$ is a CU-labeling.
\end{proposition}
\begin{proof}Let $B_1 = B\sqcup \{u\}, B_2 = B\sqcup \{w\} \in \text{Bic}(A)$ and consider the interval $[B, B_1\vee B_2]$. Recall that $B_1 \vee B_2 = B \sqcup \overline{\{u,w\}}.$ As $A$ is a brick gentle algebra, Proposition~\ref{Ind-Brick Finite} implies that $\overline{\{u,w\}} \subset \{u, w, u\alpha^{\pm 1} w, w\beta^{\pm 1} u\}$ for some $\alpha, \beta \in Q_1$ assuming both $u\alpha^{\pm 1} w$ and $w\beta^{\pm 1} u$ are strings of $A$.

If neither $u\alpha^{\pm 1} w$ nor $w\beta^{\pm 1} u$ is a string, then $[B,B_1\vee B_2]$ is the interval shown on the left in Figure~\ref{polygons_fig}. Now suppose only one of $u\alpha^{\pm 1} w$ and $w\beta^{\pm 1} u$ is a string. Without loss of generality, assume that $u\alpha^{\pm 1} w$ is a string. Then $[B,B_1\vee B_2]$ is shown on the right in Figure~\ref{polygons_fig}. Lastly, suppose that both $u\alpha^{\pm 1} w$ and $w\beta^{\pm 1} u$ are strings. Then $[B,B_1\vee B_2]$ is shown on the bottom in Figure~\ref{polygons_fig}. Using these figures, one deduces axioms (CN1), (CN2), and (CN3).

We now verify axiom (CU2), and axiom (CU1) is an immediate consequence of (CU2).
	
(CU2): Consider two meet-irreducibles $M_1, M_2 \in \text{MI}(\text{Bic}(A))$ which are covered by $M_1^*$ and $M_2^*$, respectively. Assume for the sake of contradiction that ${\lambda}(M_1, M_1^*)={\lambda}(M_2, M_2^*),$ and denote this label by $w_\mathcal{D}$. Thus $M_1^*=M_1\sqcup\{{w}\}$ and $M_2^*=M_2\sqcup\{{w}\}$. Note that $w \in M_1\vee M_2$ so there exists $u^1, \ldots, u^\ell \in M_1 \cup M_2$ and $\alpha_1, \ldots, \alpha_{\ell - 1} \in Q_1$ such that $w = u^1\alpha_1^{\pm 1}u^2 \cdots u^{\ell-1}\alpha_{\ell-1}^{\pm 1}u^\ell$. 

If there exists $i \in \{1,\ldots, \ell-1\}$ such that $u^i, u^{i+1} \in M_1$ (resp., $u^i,u^{i+1} \in M_2$), then $u^i\alpha_i^{\pm 1}u^{i+1} \in M_1$ (resp., $u^i\alpha_i^{\pm 1}u^{i+1} \in M_2$). Therefore, we can assume that the expression $u^1\alpha_1^{\pm 1}u^2 \cdots u^{\ell-1}\alpha_{\ell-1}^{\pm 1}u^\ell$ has the property that for any $i \in \{1, \ldots, \ell-1\}$ if $u^i \in M_1$ (resp., $u^i \in M_2$), then $u^{i+1} \in M_2$ (resp., $u^{i+1}\in M_1$). We can further assume, without loss of generality, that $u^1 \in M_1$.

Next, since ${\lambda}(M_1, M_1^*)={\lambda}(M_2, M_2^*)$, sets $M_1$ and $M_2$ both contain the same split of $w$ from a given break. We know that $u^1$ is a split of $w$ so $u^1 \in M_1 \cap M_2$. Since $u^2 \in M_2$, we know $u^1 \alpha_1^{\pm 1} u^2 \in M_2$. Now $u^1\alpha_1^{\pm 1}u^2$ is a split of $w$ so it follows that $u^1\alpha_1^{\pm 1} u^2 \in M_1\cap M_2$. By continuing this argument, we obtain that $w \in M_1$, a contradiction.
\end{proof}

\begin{figure}[!htbp]
$$\vcenteredhbox{\includegraphics[scale=1.2]{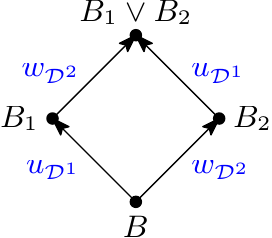}} \ \ \ \ \ \ \vcenteredhbox{\includegraphics[scale=1.2]{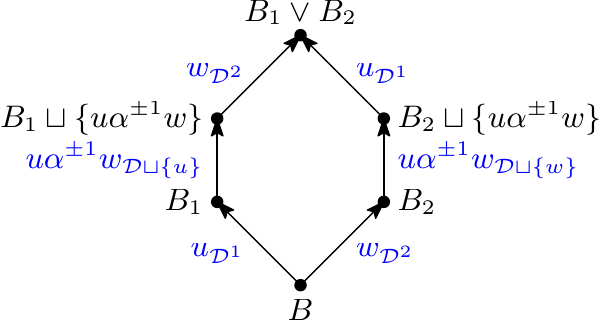}}$$
$$\vcenteredhbox{\includegraphics[scale=1.2]{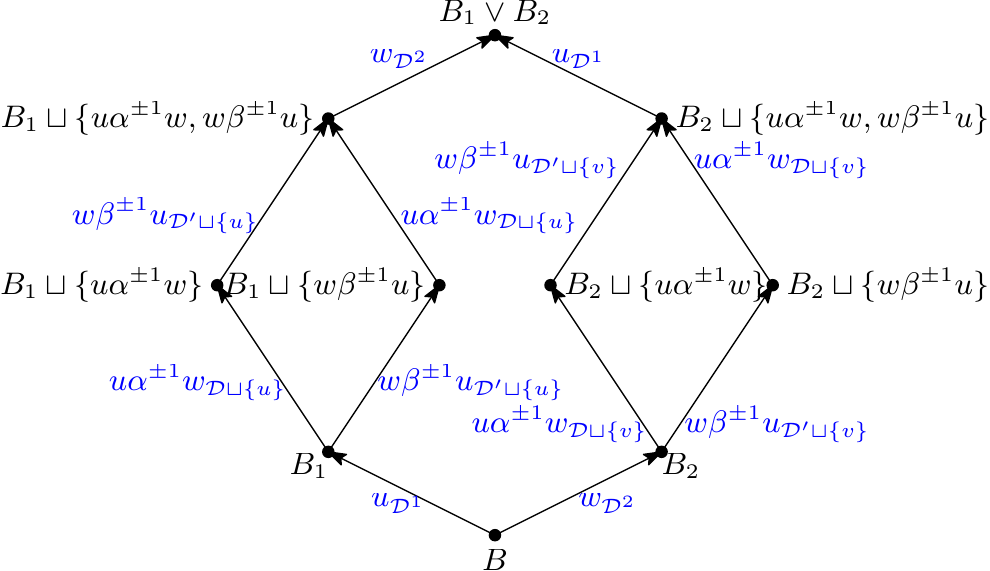}}$$
\caption{{The three forms of the interval $[B,B_1\vee B_2]$ of $\text{Bic}(A)$. The labels on the covering relations as defined by the labeling ${\lambda}: \text{Cov}(\text{Bic}(A)) \to \mathcal{S}$ are in blue. The set $\mathcal{D}^1$ (resp., $\mathcal{D}^2$) consists of all splits of $u$ (resp., $w$) belonging to $B$. Similarly, the set $\mathcal{D}$ (resp., $\mathcal{D}^\prime$) consists of all splits of $u\alpha^{\pm 1}w$ (resp., $w\beta^{\pm 1}u$) that belong to $B$.}}
\label{polygons_fig}
\end{figure}

As an application of the proof of Proposition~\ref{Prop_CU_labeling_bic}, we can say exactly which lattices of biclosed sets of strings are polygonal. A finite lattice $L$ is a \textit{polygon} if it consists of exactly two maximal chains and those chains agree only at the top and bottom elements. By definition, a finite lattice $L$ is \textit{polygonal} if for all $x \in L$ the following properties hold:
\begin{itemize}
\item if $y, z \in L$ are distinct elements covering $x$, then $[x,y\vee z]$ is a polygon, and
\item if $y, z \in L$ are distinct elements covered by $x$, then $[y\wedge z, x]$ is a polygon.
\end{itemize}

\begin{corollary}
Let $A = kQ/I$ be a brick gentle algebra. The lattice $\Bic(A)$ is polygonal if and only if there are no oriented 2-cycles in $Q$.
\end{corollary} 

\begin{proof}
Since $\Bic(A)$ is self-dual, it is polygonal if and only if every interval $[B, B_1\vee B_2]$ is a polygon where $B_1$ and $B_2$ are two distinct biclosed sets covering a biclosed set $B$. In the proof of Proposition~\ref{Prop_CU_labeling_bic}, we classified all intervals $[B, B_1\vee B_2]$ where $B_1$ and $B_2$ are two distinct biclosed sets covering a biclosed set $B$. All such intervals are polygons if and only if there are no oriented 2-cycles in $Q$.
\end{proof}

We conclude this section by classifying the join-irreducbile biclosed sets. Given $w_\mathcal{D} \in \mathcal{S}$, define $$J(w_\mathcal{D}) := \overline{\{w\} \sqcup \mathcal{D} \sqcup \bigcup_{u \in \mathcal{D}} S(u)}$$ where $S(u) = S(u,\mathcal{D}) \subset \text{Str}(A)$ is defined to be the set of all splits $v$ of $u$ satisfying the following: \begin{itemize}
\item[i)] string $v$ is not a split of $w$, and
\item[ii)] string $v$ may not be concatenated with any string in $\mathcal{D}$.
\end{itemize} 
\noindent Observe that any element of $J(w_\mathcal{D})\backslash\left(\{w\} \sqcup \mathcal{D} \sqcup \bigcup_{u \in \mathcal{D}} S(u)\right)$ is not a substring of $w$.

\begin{example}\label{example_of_set_J}
Let $A = kQ/\langle \alpha\delta, \delta\gamma\rangle$ where $Q$ is the quiver shown in Figure~\ref{J_example}. Observe that $J(\alpha^{-1}\beta\gamma^{-1}_{\{e_1, e_4, \alpha^{-1}\}}) = \{e_1, e_4, \alpha^{-1}, \delta, \alpha^{-1}\beta\gamma^{-1}\}.$
\end{example}
\begin{figure}[!htbp]
$${\xymatrix{
   2  & 3 \ar_\beta[l] \ar^\gamma[d]\\
   1 \ar^{\alpha}[u]  & 4 \ar^\delta[l]}}$$
\caption{The quiver from Example~\ref{example_of_set_J}.}
\label{J_example}
\end{figure}
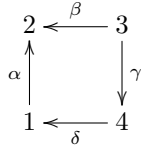

\begin{lemma}\label{lemma_J_biclosed}
The set $J(w_\mathcal{D})$ is biclosed and $\mathcal{D}$ is exactly the set of splits of $w$ contained in $J(w_\mathcal{D})$. Additionally, for any element of $\{w\} \sqcup \mathcal{D} \sqcup \bigcup_{u\in \mathcal{D}} S(u)$ exactly one element from each of its breaks belongs to $J(w_\mathcal{D})$. 
\end{lemma}

\begin{proof}
By definition, the set is closed so we show that $J(w_\mathcal{D})$ is coclosed. The proof of \cite[Lemma 3.6]{cliftondillerygarver} implies that the set $\{w\} \sqcup \mathcal{D} \sqcup \bigcup_{u \in \mathcal{D}}S(u)$ is coclosed. Thus, to complete the proof, we show that for any $u \in J(w_\mathcal{D})\backslash \left(\{w\} \sqcup \mathcal{D} \sqcup \bigcup_{u \in \mathcal{D}}S(u)\right)$ at least one element of each break of $u$ belongs to $J(w_\mathcal{D})$. To do so, suppose $w^1\alpha_1^{\pm 1}w^2 \cdots w^{k-1}\alpha_{k-1}^{\pm 1}w^k \in J(w_\mathcal{D})$ where $w^1, \ldots, w^k \in \{w\}\sqcup \mathcal{D} \sqcup \bigcup_{u\in \mathcal{D}} S(u)$ and $\alpha_1,\ldots, \alpha_{k-1} \in Q_1$. Now assume $w^1\alpha_1^{\pm 1}w^{2}\cdots w^{k-1}\alpha_{k-1}^{\pm 1}w^k = u\alpha^{\pm 1}v$ for some strings $u, v \in \text{Str}(A)$ and some $\alpha \in Q_1$. Either $u = w^1\alpha_1^{\pm 1}w^2\cdots w^{i-1}\alpha_i^{\pm 1}w^i$ and $v= w^{i+1}\alpha_{i+1}^{\pm 1}w^{i+2} \cdots w^{k-1}\alpha_{k-1}^{\pm 1}w^k$ for some $i \in \{1,\ldots, k-1\}$ or $u = w^1\alpha_1^{\pm 1}w^2\cdots w^{i-1}\alpha_{i-1}^{\pm 1}u^i$ and $v = v^{i}\alpha_i^{\pm 1}w^{i+1}\cdots w^{k-1}\alpha_{k-1}^{\pm 1}w^k$ for some $i \in \{1, \ldots, k\}$ where $u^i$ and $v^i$ are nonempty strings satisfying $u^i\beta_i^{\pm 1}v^i = w^i$ for some $\beta_i \in Q_1$.

It is enough to assume we are in the latter case. Since $\{w\}\sqcup \mathcal{D} \sqcup \bigcup_{u\in \mathcal{D}} S(u)$ is coclosed, given $w^i$ one has that $u^i \in J(w_\mathcal{D})$ or $v^i \in J(w_\mathcal{D})$. Suppose without loss of generality that $u^i \in J(w_\mathcal{D})$. As $w^1\alpha_1^{\pm 1}w^2\cdots w^{i-2}\alpha_{i-2}^{\pm 1}w^{i-1} \in J(w_\mathcal{D})$, we know $u = w^1\alpha_1^{\pm 1}w^2\cdots w^{i-1}\alpha_{i-1}^{\pm 1}u^i \in J(w_\mathcal{D})$. We obtain that $J(w_\mathcal{D})$ is coclosed.

To prove that any split of $w$ belonging to $J(w_\mathcal{D})$ belongs to $\mathcal{D}$, suppose $w = w^1\alpha_1^{\pm 1}w^2$ where $w^1, w^2 \in J(w_\mathcal{D})$ and $\alpha_1 \in Q_1$. Without loss of generality, assume $w^2 \not \in \mathcal{D}$. This implies that $w^2 = u^1\beta_1^{\pm 1}u^2\cdots u^{k-1}\beta_{k-1}^{\pm 1}u^k$ with $k \ge 2$ for some strings $u^1, \ldots, u^{k} \in \{w\} \sqcup \mathcal{D}\sqcup \bigcup_{u \in \mathcal{D}}S(u)$ and some arrows $\beta_1,\ldots, \beta_{k-1}\in Q_1$. Moreover, $u^i \in \bigcup_{u \in \mathcal{D}}S(u)$ for each $i \in \{1, \ldots, k-1\}$ and $u^k \in \mathcal{D}$. However, this implies that $u^{k-1}$ and $u^k$ may be concatenated, which contradicts that $u^{k-1} \in S(u)$ for some $u \in \mathcal{D}$. 

The final assertion is clear.
\end{proof}

We use the sets $J(w_\mathcal{D})$ to classify the join-irreducible biclosed sets in the following proposition.

\begin{proposition}\label{Prop_join_irr_des}
The biclosed set $J(w_\mathcal{D})$ satisfies ${\lambda}_\downarrow(J(w_\mathcal{D})) = \{w_\mathcal{D}\}.$ Moreover, any biclosed set $B$ with $w_\mathcal{D} \in {\lambda}_\downarrow(B)$ satisfies $J(w_\mathcal{D}) \le B$, and the reverse containment holds if and only if ${\lambda}_{\downarrow}(B) = \{w_\mathcal{D}\}$. Consequently, the set map $J(-): \mathcal{S} \to \JI(\Bic(A))$ is a bijection.
\end{proposition}
\begin{proof}
Since $\mathcal{D}$ is exactly the set of splits of $w$ contained in $J(w_\mathcal{D})$, $w$ is not expressible as a concatenation of elements of $J(w_\mathcal{D})$. This implies that $J(w_\mathcal{D})\backslash\{w\}$ is biclosed. Moreover, $w_\mathcal{D} \in {\lambda}_\downarrow(J(w_\mathcal{D}))$. 

Now let $v_{\mathcal{D}^\prime} \in {\lambda}_{\downarrow}(J(w_\mathcal{D}))$. The string $v \not \in J(w_\mathcal{D})\backslash\left(\{w\} \sqcup \mathcal{D}\sqcup \bigcup_{u \in \mathcal{D}}S(u)\right)$, otherwise $J(w_\mathcal{D})\backslash\{v\}$ is not closed. If $v \in \mathcal{D}$, then by Lemma~\ref{lemma_J_biclosed} there is not a split of $w$ in $J(w_\mathcal{D})\backslash\{v\}$ from each break of $w$. Therefore $J(w_\mathcal{D})\backslash\{v\}$ is not closed. 

Next, suppose $v \in S(u)$ for some $u \in \mathcal{D}$. Since $v \in S(u)$, writing $u = v^\prime\alpha^{\pm 1} v$ for some $\alpha \in Q_1$ implies that $v, v^\prime \not \in J(w_\mathcal{D})\backslash\{v\}$. Thus $J(w_\mathcal{D})\backslash\{v\}$ is not coclosed. Consequently, $v = w$ and Lemma~\ref{lemma_J_biclosed} therefore implies that $\mathcal{D} = \mathcal{D}^\prime.$ 

Now assume that $w_\mathcal{D} \in {\lambda}_\downarrow(B)$ for some biclosed set $B \in {\Bic}(A)$. Since the set of splits of $w$ contained in $B$ is the set $\mathcal{D}$, it is clear that $\{w\} \sqcup \mathcal{D} \sqcup \bigcup_{u \in \mathcal{D}} S(u) \subset B$. The set $B$ is closed so we conclude that $J(w_\mathcal{D}) \le B$. 

We have shown that $J(w_\mathcal{D})$ is the minimal biclosed set satisfying ${\lambda}_\downarrow(J(w_\mathcal{D}))$ so by Lemma~\ref{cu-label_ji_mi} we obtain the remaining assertions.\end{proof}

\section{Torsion shadows}\label{torshad_section}

In this section, we show that the data of a biclosed subcategory of the module category of a brick gentle algebra $A$ is equivalent to a certain subcategory of the module category of an algebra analogous to a preprojective algebra. This algebra will be denoted by $\Pi(A)$, and we refer to the relevant subcategories of $\text{mod}(\Pi(A))$ as torsion shadows. 

Recall that in the gentle bound quiver of $A=kQ/I$, every generator of $I$ is given by a pair of arrows $\alpha$ and $\beta$ such that $\beta \alpha$ is a path of length two in $Q$.
Let $\overline{Q}$ be the \emph{doubled quiver} of $Q$ (i.e., $\overline{Q}_0:=Q_0$ and $\overline{Q}_1:=Q_1 \cup Q^{\ast}_1$) and $\overline{I} := \langle \beta\alpha, \alpha^*\beta^* | \beta\alpha \in I\rangle$ the two-sided ideal in $k\overline{Q}$ determined by the relations generating $I$ and their duals. Define $\Pi(A):=k\overline{Q}/ \overline{I}$.

We now give a general definition of torsion shadows, the main examples of which will be the above mentioned subcategories of $\modu(\Pi(A))$. We also present a general lemma about torsion shadows.

\begin{definition}
Let $\mathcal{M}$ be a subcategory of $\modu(\Lambda)$. For every $\mathcal{T} \in \tors(\Lambda)$, the $\mathcal{M}$-\emph{torsion shadow} (or simply \emph{torsion shadow}) of $\mathcal{T}$ is $\mathfrak{T}_{\mathcal{M}}:= \mathcal{T} \cap \mathcal{M}.$ We let $\torshad_\mathcal{M}(\Lambda)$ denote the poset of all $\mathcal{M}$-torsion shadows in $\modu(\Lambda)$ ordered by inclusion.
\end{definition}

Provided there is no confusion, we often suppress $\mathcal{M}$ and simply use $\mathfrak{T}$ for the $\mathcal{M}$-torsion shadow of $\mathcal{T}$. 


\begin{lemma}\label{torshad-epi}
If $\Lambda$ is an algebra and $\mathcal{M}$ a subcategory of $\modu(\Lambda)$, then $\torshad_{\mathcal{M}}(\Lambda)$ forms a complete lattice and the map  $(-)\cap \mathcal{M}: \tors(\Lambda) \twoheadrightarrow \torshad_{\mathcal{M}}(\Lambda)$ is a surjective lattice map. 

If $\phi:\Lambda \twoheadrightarrow \Lambda'$ is an algebra epimorphism and $\mathcal{M}$ contains $\modu(\Lambda')$, then the map $(-)\cap \modu(\Lambda'): \torshad_{\mathcal{M}}(\Lambda) \twoheadrightarrow \tors(\Lambda')$ is a surjective lattice map. Additionally, the surjective lattice map $(-)\cap \modu(\Lambda'): \tors(\Lambda) \twoheadrightarrow \tors(\Lambda')$ factors through $\torshad_\mathcal{M}(\Lambda)$.\end{lemma}

\begin{proof}
Given a family of torsion shadows $\{\mathfrak{T}_i\}_{i \in I} \subseteq \torshad_{\mathcal{M}}(\Lambda),$ there exist torsion classes $\{\mathcal{T}_i\}_{i \in I}$ such that $\mathfrak{T}_i = \mathcal{T}_i\cap \mathcal{M}$ for all $i \in I$. By defining $\bigwedge_{i \in I} \mathfrak{T}_i := \bigcap_{i \in I} \mathfrak{T}_i$ and the fact that $\bigcap_{i \in I} \mathcal{T}_i \in \tors(\Lambda)$, it is clear that $\torshad_{\mathcal{M}}(\Lambda)$ is a complete meet-semilattice. Since $\modu(\Lambda) \cap \mathcal{M} \in \torshad_{\mathcal{M}}(\Lambda)$ is the unique maximal element of $\torshad_{\mathcal{M}}(\Lambda)$, we obtain that $\torshad_{\mathcal{M}}(\Lambda)$ is a complete lattice. 

It is straightforward to show that the maps $(-)\cap \mathcal{M}: \tors(\Lambda) \twoheadrightarrow \torshad_{\mathcal{M}}(\Lambda)$ and $(-)\cap \modu(\Lambda'): \torshad_{\mathcal{M}}(\Lambda) \twoheadrightarrow \tors(\Lambda')$ are surjective meet-semilattice maps. We show that $(-)\cap \modu(\Lambda'): \torshad_{\mathcal{M}}(\Lambda) \twoheadrightarrow \tors(\Lambda')$ is a join-semilattice map. The proof that $(-)\cap \mathcal{M}: \tors(\Lambda) \twoheadrightarrow \torshad_{\mathcal{M}}(\Lambda)$ is a join-semilattice map is similar so we omit it.

Let $\{\mathcal{T}_i\cap \mathcal{M}\}_{i\in I} \subset \torshad_\mathcal{M}(\Lambda)$ be a family of torsion shadows. We have 
$$\begin{array}{rcl}
\displaystyle \left(\bigvee_{i \in I}\mathcal{T}_i\cap \mathcal{M} \right) \cap \modu(\Lambda') & = & \displaystyle \left(\mathop{\bigcap_{\mathcal{T}\cap \mathcal{M} \in \torshad_{\mathcal{M}}(\Lambda)}}_{\mathcal{T}_i \subseteq \mathcal{T} \ \forall i \in I}\mathcal{T} \cap \mathcal{M}\right) \cap \modu(\Lambda') \\
& = & \displaystyle \left(\mathop{\bigcap_{\mathcal{T}\in \tors(\Lambda)}}_{\mathcal{T}_i \subseteq \mathcal{T} \ \forall i \in I}\mathcal{T} \cap \mathcal{M}\right) \cap \modu(\Lambda') \\
& = & {\displaystyle \mathop{\bigcap_{\mathcal{T}\in \tors(\Lambda)}}_{\mathcal{T}_i \subseteq \mathcal{T} \ \forall i \in I}\mathcal{T} \cap \modu(\Lambda')} \\
& & \text{(using that $\modu(\Lambda') \subseteq \mathcal{M}$)}\\
& = & {\displaystyle \mathop{\bigcap_{\mathcal{T}\in \tors(\Lambda')}}_{\mathcal{T}_i \cap \modu(\Lambda') \subseteq \mathcal{T} \ \forall i \in I}\mathcal{T} \cap \modu(\Lambda')} \\
& = & \displaystyle \bigvee_{i \in I} \mathcal{T}_i \cap \modu(\Lambda'). \\
\end{array}$$
This shows that $(-)\cap \modu(\Lambda'): \torshad_{\mathcal{M}}(\Lambda) \twoheadrightarrow \tors(\Lambda')$ is a join-semilattice map. 

It is clear that the map $(-)\cap \modu(\Lambda'): \tors(\Lambda) \twoheadrightarrow \tors(\Lambda')$ factors through $\torshad_\mathcal{M}(\Lambda)$.\end{proof}





We now focus on brick gentle algebras $A$ and the associated algebras $\Pi(A)$. For the remainder of the section $A$ denotes a brick gentle algebra. We write an arbitrary arrow of $\Pi(A)$ as $\overline{\gamma}$ where $\overline{\gamma}=\gamma$ or $\overline{\gamma}=\gamma^*$ for some $\gamma \in Q_1$. Let $\widetilde{\Str}(A)$ denote the set of strings $\widetilde{w}=\overline{\gamma}_d^{\epsilon_d} \cdots \overline{\gamma}_2^{\epsilon_2}\overline{\gamma}_1^{\epsilon_1} \in \Str(\Pi(A))$ where $\widetilde{w}$ \textit{specializes} to a string of $A$ (i.e., the sequence of arrows of $Q$ and formal inverses of arrows of $Q$ obtained by replacing every $\gamma_i^*$ in $\widetilde{w}$ by $\gamma_i^{-1}$ and every $(\gamma_i^*)^{-1}$ with $\gamma_i$ is a string in $\Str(A)$). We set $$\mathcal{M}:= \add \bigg( \bigoplus M ( \widetilde{w}) |\, \widetilde{w} \in \widetilde{\Str}(A)\bigg)$$ and for the remainder of the paper, unless specified otherwise, for every brick gentle algebra $A$ we let $\mathcal{M}$ denote this subcategory of $\modu(\Pi(A))$.

Given any string in $\Str(A)$, we can lift it to a string in $\widetilde{\Str}(A)$. First, choose whether to represent the string as $w={\gamma}_d^{\epsilon_d} \cdots {\gamma}_2^{\epsilon_2}{\gamma}_1^{\epsilon_1}$ or as $w^{-1}={\gamma}_1^{-\epsilon_1} \cdots {\gamma}_2^{-\epsilon_{d-1}}{\gamma}_d^{-\epsilon_d}$. Then, replace every $\gamma^{-1}$ with $\gamma^{\ast}$, and let $\widetilde{w}$ denote the resulting string in $\widetilde{\Str}(A)$. By Proposition \ref{Ind-Brick Finite}, every string $\widetilde{w} \in \Str(\Pi(A))$ constructed in this way is self-avoiding. For any $\mathcal{T} \in \tors(\Pi(A))$, we let $\widetilde{\mathfrak{T}}:=\mathcal{T} \cap \mathcal{M}$ denote the corresponding torsion shadow, and $\mathfrak{T} := \text{add}(\bigoplus M(w) \,|\, M(\widetilde{w}) \in \widetilde{\mathfrak{T}}).$

%

We can now state one of the main theorems of this section.

\begin{theorem}\label{Bic-Torshad}
There is a poset isomorphism $\Bic(A) \cong \torshad(\Pi(A))$ and $$\bic(A)= \{\mathfrak{T}\,|\, \widetilde{\mathfrak{T}} \in \torshad(\Pi(A))\}.$$
\end{theorem}

The proof of the theorem is a consequence of the lemmas and proposition that we now prove. To state these results, we define two maps $$\widetilde{\mathfrak{T}}(-):\Bic(A) \rightarrow \torshad(\Pi(A)) \qquad \text{and} \qquad B(-): \torshad(\Pi(A))\rightarrow \Bic(A).$$ 

We first define the map $\widetilde{\mathfrak{T}}(-)$ by connecting the CU-labeling of $\Bic(A)$ with the strings in $\widetilde{\Str}(A)$. Each label $w_{\mathcal{D}} = w_{\{w^1,\ldots, w^d\}} \in \mathcal{S}$ with $w = \gamma_d^{\epsilon_d}\cdots \gamma_1^{\epsilon_1}$ gives rise to a string $\text{str}(w_{\mathcal{D}}) = \overline{\gamma_d}^{\epsilon_d^\prime}\cdots \overline{\gamma_1}^{\epsilon_1^\prime} \in \widetilde{\Str}({A})$. Define $\overline{\gamma_i}^{\epsilon_i^\prime} := \gamma_i$ (resp., $\overline{\gamma_i}^{\epsilon_i^\prime} = \gamma_i^*$) if $w = u\gamma_iw^j$ (resp., if $w = u\gamma_i^{-1}w^j$) for some $j \in \{1, \ldots, d\}$ and some (possibly empty) string $u \in \Str(A)$. Similarly, define $\overline{\gamma_i}^{\epsilon_i^\prime} := \gamma_i^{-1}$ (resp., $\overline{\gamma_i}^{\epsilon_i^\prime} := (\gamma_i^{*})^{-1}$) if $w = w^j\gamma_i^{-1}u$ (resp., $w = w^j\gamma_i u$) for some $j \in \{1,\ldots, d\}$ and some (possibly empty) string $u \in \Str(A)$. Recall that, by the definition of $w_{\{w^1,\ldots, w^d\}}$, no two strings $w^j, w^{j^\prime} \in \{w^1, \ldots, w^d\}$ satisfy $w = w^j\gamma^{\pm 1} w^{j^\prime}$ for any $\gamma \in Q_1$. Therefore, the map $\text{str}(-): \mathcal{S}\to \widetilde{\Str}(A)$ is well-defined.

The following lemma is easily verified.


\begin{lemma}\label{Lemma_labels_to_Pi_strings}
The map $\str(-): \mathcal{S} \to \widetilde{\Str}(A)$ sending a label to the corresponding string in $\Pi(A)$ is a bijection. 
\end{lemma}

We also state the following lemma which shows that every string module defined by a string in $\widetilde{\Str}(A)$ is brick. 

\begin{lemma}\label{bricks_in_Pi}
Given any label $w_\mathcal{D}$, the string module $M(\text{str}(w_{\mathcal{D}})) \in \mathcal{M}$ is a brick as a $\Pi(A)$-module. 
\end{lemma}
\begin{proof}
Let $\varphi = (\varphi_i)_{i \in \overline{Q}_0} \in \text{End}_{\Pi(A)}(M(\text{str}(w_{\mathcal{D}})))$ be an endomorphism of the quiver representation $M(\text{str}(w_\mathcal{D}))$. As the string $\text{str}(w_{\mathcal{D}})$ visits a vertex of $\overline{Q}$ at most once, each linear map $\varphi_i$ is a scalar transformation. One checks that there exists $\lambda \in k$ such that $\varphi_i = \lambda\text{id}_k$ for all vertices $i$ appearing in $\text{str}(w_\mathcal{D})$. We obtain that $\text{End}_{\Pi(A)}(M(\text{str}(w_{\mathcal{D}}))) = k$.
\end{proof}

Now, given $B \in \Bic(A)$, define $\widetilde{\mathfrak{T}}(B) := \mathcal{T}^* \cap \mathcal{M}$ where $\mathcal{T}^*$ is the minimal torsion class in $\modu(\Pi(A))$ that contains $\text{gen}\left(\bigoplus M(\text{str}(w_\mathcal{D})) \ | \ w_\mathcal{D} \in \lambda_\downarrow(B)\right)$. By definition, $\widetilde{\mathfrak{T}}(B)$ is a torsion shadow of $A$. Moreover, we have the following explicit description of $\widetilde{\mathfrak{T}}(B)$.

\begin{lemma}\label{lemma_desc_T(B)}
Let $B \in \Bic(A)$ be a biclosed set. The indecomposable objects of $\widetilde{\mathfrak{T}}(B)$ are exactly the string modules $M(\widetilde{w}) \in \mathcal{M}$ all of whose indecomposable quotients are string modules $M(\widetilde{u}) \in \mathcal{M}$ where $\widetilde{u}$ specializes to a string $u \in B$. In addition, for any element $u \in B$, there exists a string module $M(\widetilde{u}) \in \widetilde{\mathfrak{T}}(B)$ such that $\widetilde{u}$ specializes to $u$.
\end{lemma}

\begin{example}
Let $Q$ be the quiver appearing in Figure~\ref{J_example}, and let $J(\alpha^{-1}\beta\gamma_{\{e_1,e_4,\alpha\}})$ be the join-irreducible biclosed set from Example~\ref{example_of_set_J}. Here we have that $$\widetilde{\mathfrak{T}}(J(\alpha^{-1}\beta\gamma_{\{e_1,e_4,\alpha\}})) = \add(\bigoplus M(w) \ | \ w \in \{e_1,e_4, \alpha^{-1}, \delta, \delta^*, \alpha^{-1}(\beta^*)^{-1}\gamma^*\}).$$ Here $\text{str}(\alpha^{-1}\beta\gamma_{\{e_1,e_4,\alpha\}}) = \alpha^{-1}(\beta^*)^{-1}\gamma^*$.
\end{example}

\begin{proof}
We can write $B = \bigvee_i J(w^i_{\mathcal{D}^i})$ where the join is taken over all $w^i_{\mathcal{D}^i} \in \lambda_\downarrow(B).$ Observe that from the definition of $\text{str}(-)$, for each $i$, every indecomposable quotient of $M(\text{str}(w^i_{\mathcal{D}^i}))$ is a string module whose string specializes to an element of $B$. We also know from the definition of $\widetilde{\mathfrak{T}}(B)$ that $M(\text{str}(w^i_{\mathcal{D}^i})) \in \widetilde{\mathfrak{T}}(B)$ for all $i$. It therefore follows that for any element of $u \in J(w^i_{\mathcal{D}^i})$ there exists a string module $M(\widetilde{u}) \in \widetilde{\mathfrak{T}}(B)$ such that $\widetilde{u}$ specializes to $u$. Moreover, any such $M(\widetilde{u})$ has the desired property since $M(\text{str}(w^i_{\mathcal{D}^i}))$ has the desired property. 

Now, we see that an arbitrary indecomposable object of $\widetilde{\mathfrak{T}}(B)$ is a string module $M(\widetilde{u})$ where $\widetilde{u} = \widetilde{u}^{i_1}\overline{\gamma_{i_1}}^{\pm 1}\widetilde{u}^{i_2}\cdots \widetilde{u}^{i_{k-1}}\overline{\gamma_{i_{k-1}}}^{\pm 1}\widetilde{u}^{i_k}$ where each $\widetilde{u}^{i_j}$ specializes to a string $u^{i_j} \in J(w^{i_j}_{\mathcal{D}^{i_j}})$ and where there is a surjection $M(\text{str}(w^{i_j}_{\mathcal{D}^{i_j}})) \twoheadrightarrow M(\widetilde{u}^{i_j})$ for all $j$. Since each module $M(\widetilde{u}^{i_j})$ has the desired property, the module $M(\widetilde{u})$ does as well.

Conversely, suppose $M(\widetilde{w})$ is any string module with the property in the statement of the lemma. Since $\widetilde{w}$ specializes to a string $w \in B$, we know that $w = u^{i_1}{\gamma_{i_1}}^{\pm 1}{u}^{i_2}\cdots {u}^{i_{k-1}}{\gamma_{i_{k-1}}}^{\pm 1}{u}^{i_k}$ where $u^{i_j} \in J(w^{i_j}_{\mathcal{D}^{i_j}})$ for all $j$. We thus have $\widetilde{w} = \widetilde{u}^{i_1}\overline{\gamma_{i_1}}^{\pm 1}\widetilde{u}^{i_2}\cdots \widetilde{u}^{i_{k-1}}\overline{\gamma_{i_{k-1}}}^{\pm 1}\widetilde{u}^{i_k}$ where $\widetilde{u}^{i_j}$ specializes to $u^{i_j} \in J(w^{i_j}_{\mathcal{D}^{i_j}})$ for all $j$.

It remains to show that for each $j$, there is a surjection $M(\text{str}(w^{i_j}_{\mathcal{D}^{i_j}})) \twoheadrightarrow M(\widetilde{u}^{i_j})$. We prove this by induction on $k$. If $k = 1$, then $\widetilde{w} = \widetilde{u}^{i_1}$. Here, the result follows from the fact that $M(\widetilde{u}^{i_1})$ has the property in the statement of the lemma. 

Next, suppose that $k > 1$ and that the result holds for all $k^\prime < k$. Observe that there exists $j$ such that $M(\widetilde{u}^{i_j})$ is a submodule of $M(\widetilde{w})$. Now apply the inductive hypothesis to the modules $$\begin{array}{l}M(\widetilde{u}^{i_1}\overline{\gamma_1}^{\pm 1} \widetilde{u}^{i_2}\cdots \widetilde{u}^{i_{j-2}}\overline{\gamma_{j-2}}^{\pm 1}\widetilde{u}^{i_{j-1}}),\\
M(\widetilde{u}^{i_{j+1}}\overline{\gamma_{j+1}}^{\pm 1} \widetilde{u}^{i_{j+2}}\cdots \widetilde{u}^{i_{k-1}}\overline{\gamma_{k-1}}^{\pm 1}\widetilde{u}^{i_{k}}),\\
M(\widetilde{u}^{i_j}) \end{array}$$ to obtain that each belongs to $\widetilde{\mathfrak{T}}(B)$. Since ${\mathcal{T}}^*$ is extension-closed and $M(\widetilde{w}) \in \mathcal{M}$, we now have that $M(\widetilde{w}) \in \widetilde{\mathfrak{T}}(B)$.\end{proof}

Next, let $\widetilde{\mathfrak{T}} \in \torshad(\Pi(A))$ be given. By Lemma~\ref{Lemma_labels_to_Pi_strings}, we have that for any indecomposable object $M \in \widetilde{\mathfrak{T}}$ there is a unique label $w_\mathcal{D} \in \mathcal{S}$ such that $M\cong M(\text{str}(w_\mathcal{D}))$. We define $B(\widetilde{\mathfrak{T}}) := \{w \in \Str(A) \ | \ M(\text{str}(w_\mathcal{D})) \in \widetilde{\mathfrak{T}}\}.$


\begin{lemma}\label{biclosed map}
For any $\widetilde{\mathfrak{T}} \in \torshad(\Pi(A))$, the set of strings $B(\widetilde{\mathfrak{T}})$ is a biclosed set.
\end{lemma}
\begin{proof}
Let $\widetilde{\mathfrak{T}}$ be a torsion shadow with $\widetilde{\mathfrak{T}}= \mathcal{T}\cap \mathcal{M}$, for some $\mathcal{T} \in \tors(\Pi(A))$. Suppose $w, w' \in B(\widetilde{\mathfrak{T}})$ and let $M(\text{str}(w_\mathcal{D}))$ and $M(\text{str}(w'_{\mathcal{D}'}))$ denote the corresponding indecomposables in $\widetilde{\mathfrak{T}}$.

Assuming that $w\gamma w' \in \Str(A)$ where $\gamma$ is some arrow of $Q$, we show that $w\gamma w' \in B(\widetilde{\mathfrak{T}})$. The proof is very similar when the concatenation is of the form $w\gamma^{-1}w'$, so we omit it. By assumption, if $w\gamma w' \in \Str(A)$, then $\widetilde{u} = \text{str}(w_\mathcal{D})\gamma \text{str}(w'_{\mathcal{D}'}) \in \widetilde{\Str}(A)$. Observe that there is an extension $$0 \to M(\text{str}(w_\mathcal{D})) \to M(\widetilde{u}) \to M(\text{str}(w'_{\mathcal{D}'})) \to 0$$ in $\modu(\Pi(A))$. Since $\mathcal{T}$ is extension-closed, we have that $M(\widetilde{u}) \in \mathcal{T}$. We obtain that $w\gamma w' =u \in B(\widetilde{\mathfrak{T}})$. Therefore, $B(\widetilde{\mathfrak{T}})$ is closed.

Next, we prove that $B(\widetilde{\mathfrak{T}})$ is coclosed. Assume $w \in B(\widetilde{\mathfrak{T}})$ and that $w=v\gamma v'$ for some strings $v$ and $v'$ in $\Str(A)$ and some arrow $\gamma \in Q_1$. The proof is very similar when we assume that $w = v\gamma^{-1}v'$ so we omit it. Let $M(\widetilde{w}) \in \widetilde{\mathfrak{T}}$ denote a string module that specializes to $w$. We know that $\widetilde{w} = \text{str}(v_{\mathcal{D}})\gamma^{\pm 1}\text{str}(v'_{\mathcal{D}'})$ or $\widetilde{w} = \text{str}(v_{\mathcal{D}})({\gamma^*})^{\pm 1}\text{str}(v'_{\mathcal{D}'})$. Without loss of generality, there is a surjection $M(\widetilde{w})\twoheadrightarrow M(\text{str}(v'_{\mathcal{D}'}))$. Since $\widetilde{\mathfrak{T}}$ is a quotient-closed, we have that $M(\text{str}(v'_{\mathcal{D}'})) \in \widetilde{\mathfrak{T}}$. By the definition of $B(\widetilde{\mathfrak{T}})$, we know that $v' \in B(\widetilde{\mathfrak{T}})$. Therefore, $B(\widetilde{\mathfrak{T}})$ is coclosed.\end{proof}


\begin{proposition}\label{Prop_inverses}
We have the following identities:
\begin{enumerate} [(i)]
\item $\widetilde{\mathfrak{T}}=\widetilde{\mathfrak{T}}(B(\widetilde{\mathfrak{T}}))$ for all $\mathfrak{T} \in \torshad(\Pi(A))$;
\item $B=B(\widetilde{\mathfrak{T}}(B))$ for all $B \in \Bic(A)$.
\end{enumerate}
\end{proposition}
\begin{proof}
To prove $(i)$, first, assume $M(\widetilde{w}) \in \text{ind}(\widetilde{\mathfrak{T}})$. Since $\widetilde{\mathfrak{T}}$ is quotient-closed, we know that every indecomposable quotient of $M(\widetilde{w})$ belongs to $\widetilde{\mathfrak{T}}$. By the definition of $B(-)$, we see that $\widetilde{u}$ specializes to a string $u \in B(\widetilde{\mathfrak{T}})$ where $M(\widetilde{u})$ is any indecomposable quotient of $M(\widetilde{w})$. By Lemma~\ref{lemma_desc_T(B)}, $M(\widetilde{w}) \in \text{ind}(\widetilde{\mathfrak{T}}(B(\widetilde{\mathfrak{T}})))$.

Next, write $\widetilde{\mathfrak{T}}(B(\widetilde{\mathfrak{T}})) = \mathcal{T}^*\cap \mathcal{M}$ and $\widetilde{\mathfrak{T}} = \mathcal{T}^{**}\cap \mathcal{M}$ where $\mathcal{T}^*$ is the smallest torsion class in $\modu(\Pi(A))$ that contains $\text{gen}(\bigoplus M(\text{str}(w^i_{\mathcal{D}^i}))\ | \ w^i_{\mathcal{D}^i}\in \lambda_\downarrow(B(\widetilde{\mathfrak{T}})))$ and $\mathcal{T}^{**}$ is the smallest torsion class in $\modu(\Pi(A))$ that contains $\widetilde{\mathfrak{T}}$. To prove the opposite containment, it is enough to show that $M(\text{str}(w^i_{\mathcal{D}^i})) \in \widetilde{\mathfrak{T}}$ for all $i$, since this would imply that $\mathcal{T}^* \subseteq \mathcal{T}^{**}.$ 

Now, observe that for any $i$ we have $w^i \in B(\widetilde{\mathfrak{T}})$. This implies that there exists $\mathcal{D}$ such that $M(\text{str}(w^i_\mathcal{D})) \in \widetilde{\mathfrak{T}}.$ Notice that $\mathcal{D}, \mathcal{D}^i \subseteq B(\widetilde{\mathfrak{T}}).$ If $\mathcal{D} \neq \mathcal{D}^i$, then there exists $u \in \mathcal{D}$ and $u^i \in \mathcal{D}^i$ such that $w^i = u\gamma^{\pm 1}u^i$. This equation contradicts that $w^i_{\mathcal{D}^i} \in \lambda_{\downarrow}(B(\widetilde{\mathfrak{T}})).$ This completes the proof of $(i)$.

We now prove $(ii)$. Assume $w \in B$. By Lemma~\ref{lemma_desc_T(B)}, there exists $\widetilde{w} \in \widetilde{\mathfrak{T}}(B)$ that specializes to $w$. By the definition of $B(-)$, we know $w \in B(\widetilde{\mathfrak{T}}(B))$.

To prove the opposite containment, assume $w \in B(\widetilde{\mathfrak{T}}(B))$. By the definition of $B(-)$, there exists $\mathcal{D}$ such that $M(\text{str}(w_\mathcal{D})) \in \widetilde{\mathfrak{T}}(B)$. Now Lemma~\ref{lemma_desc_T(B)} implies that $w \in B$.\end{proof}

\begin{proof}[Proof of Theorem \ref{Bic-Torshad}]
It follows from Proposition~\ref{Prop_inverses} that the maps $B(-)$ and $\widetilde{\mathfrak{T}}(-)$ are bijections. To complete the proof of the first assertion, we must show that these maps are order-preserving. If $B_1, B_2 \in \Bic(A)$ and $B_1\subseteq B_2$, then Lemma~\ref{lemma_desc_T(B)} implies that $\widetilde{\mathfrak{T}}(B_1) \subseteq \widetilde{\mathfrak{T}}(B_2)$. By definition, $B(-)$ is an order-preserving map. 

For the second part, recall that there is a bijection from $\Bic(A)$ to $\bic(A)$ which sends each biclosed set $B$ to the biclosed subcategory $\mathcal{B}:=\add ( \bigoplus M (w) |\, w \in B )$ in $\modu(A)$. Furthermore, for every $\widetilde{\mathfrak{T}} \in \torshad(\Pi(A))$, we previously defined $\mathfrak{T}:= \add ( \bigoplus M(w) \,|\, M(\widetilde{w}) \in \widetilde{\mathfrak{T}} ).$ This gives the desired identity.
\end{proof}

\begin{example}\label{Running example}
Let $A$ denote the following brick gentle algebra from Example~\ref{Ex_2cyc_alg}, and let $\Pi(A)$ denote its associated overalgebra. 

\vspace{-.3in}$$A = k(\xymatrix{
   1 \ar@<1ex>[r]^\alpha & 2 \ar@<1ex>[l]^\beta})/\langle \alpha\beta, \beta\alpha\rangle \ \ \ \Pi(A) =k(\xymatrix{
   1 \ar@<1ex>[r]^\alpha \ar@<-3ex>[r]_{\beta^*}  & 2 \ar@<1ex>[l]^\beta \ar@<-3ex>[l]_{\alpha^*}})/\langle \alpha\beta, \beta\alpha, \beta^*\alpha^*, \alpha^*\beta^*\rangle$$
 
In Figure~\ref{fig_torshad(A2_preproj)}, we show the lattice of torsion shadows of $A$. Here we describe each torsion shadow $\widetilde{\mathfrak{T}}$ simply by showing the strings defining the string modules in $\widetilde{\mathfrak{T}}$.\end{example}

\begin{figure}
$$\includegraphics[scale=1]{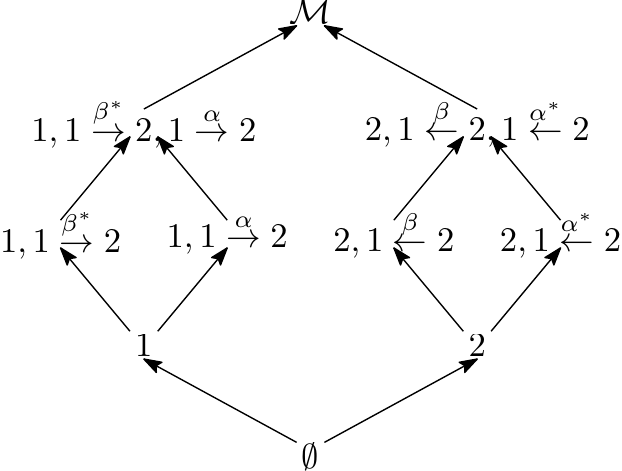}$$
\caption{A lattice of torsion shadows.}
\label{fig_torshad(A2_preproj)}
\end{figure}

%
%

\section{Wide shadows}\label{sec:wide}
Recall that a subcategory $\mathcal{W}$ of $\modu(\Lambda)$ is said to be \emph{wide} if it is abelian and closed under extensions. Let $\wide (\Lambda)$ denote the set of all wide subcategories of $\modu(\Lambda)$, ordered by inclusion. Recall that a subcategory $\mathcal{C}$ of $\modu(A)$ is functorially finite if each module $M$ in $\modu(A)$ admits right and left $\mathcal{C}$-approximations. For details on approximation theory, see \cite{AR}. Let $\fwide(\Lambda)$ (resp., $\ftors(\Lambda)$) be the subposet of $\wide(\Lambda)$ (resp., $\tors(\Lambda)$) consisting of all functorially finite wide subcategories (resp., torsion classes). 

For an acyclic quiver $Q$, Ingalls and Thomas in \cite{IT} establish several bijections between various families of representation theoretic objects associated with $kQ$. Among these is a bijection between $\ftors(kQ)$ and $\fwide(kQ)$. More recently, in \cite{MS}, Marks and \v S\v tov\` i\v cek consider the question of when $\ftors(\Lambda)$ and $\fwide(\Lambda)$ are in bijection for an arbitrary finite dimensional algebra $\Lambda$. In particular, they show that these categories are in bijection if every torsion class of $\Lambda$ is functorially finite. In this case, the bijective maps between $\ftors(\Lambda)$ and $\fwide(\Lambda)$ are the same as those discovered by Ingalls and Thomas. 

If $\Lambda$ is a representation finite algebra, then every torsion class is functorially finite. Therefore, there is a bijection between $\tors(\Lambda)$ and $\wide(\Lambda)$. In this section, given a brick gentle algebra $A$, we consider the question of whether there is a family of subcategories of $\modu(\Pi(A))$ that behave like wide subcategories and that are in bijection with the elements of $\torshad(\Pi(A))$ via maps that are analogous to those of Ingalls and Thomas and of Marks and \v S\v tov\` i\v cek. It turns out that such a family of subcategories exist; we will refer to these subcategories as \textit{wide shadows}.

\begin{definition}
Let $\mathcal{M}$ be an arbitrary full subcategory of $\modu(\Lambda)$. For every $\mathcal{W} \in \wide(\Lambda)$, the $\mathcal{M}$-\emph{wide shadow} (or simply \emph{wide shadow}) of $\mathcal{W}$ is defined as $\mathfrak{W}_{\mathcal{M}}:= \mathcal{W} \cap \mathcal{M}.$ Let $\widshad_{\mathcal{M}}(\Lambda)$ denote the poset of all $\mathcal{M}$-wide shadows ordered by inclusion.
\end{definition}

Observe that the poset $\text{wide}(\Lambda)$ is a closed under arbitrary intersections of wide subcategories. Consequently, given a wide shadow $\mathfrak{W}_\mathcal{M}$, there is a well-defined smallest wide subcategory of $\text{mod}(\Lambda)$ that contains $\mathfrak{W}_\mathcal{M}$. Therefore, when considering a particular wide shadow $\mathfrak{W}_\mathcal{M}$, we will tacitly assume that it is expressed as $\mathfrak{W}_\mathcal{M} = \mathcal{W}\cap \mathcal{M}$ where $\mathcal{W}$ is the smallest wide subcategory of $\text{mod}(\Lambda)$ containing it. 

The following lemma for wide shadows is the counterpart of Lemma~\ref{torshad-epi} for torsion shadows. 


\begin{lemma}\label{widshad-epi}
Let $\phi: B \twoheadrightarrow A$ be an algebra epimorphism and $\mathcal{M}$ a full subcategory of $\modu(B)$ which contains $\modu(A)$. Then 

\begin{enumerate}
\item the poset $\widshad_{\mathcal{M}}(B)$ is a complete lattice, and
\item the maps $(-)\cap \mathcal{M}: \wide(B) \twoheadrightarrow \widshad_{\mathcal{M}}(B)$ and $(-)\cap \modu(A):\widshad_{\mathcal{M}}(B) \twoheadrightarrow \wide(A)$ are meet-semilattice epimorphisms.
\end{enumerate}


\begin{proof}

(1) Given a family of wide shadows $\{\mathfrak{W}_i\}_{i \in I} \subseteq \widshad_{\mathcal{M}}(B)$, there exist wide subcategories
$\{\mathcal{W}_i\}_{i \in I}\in \wide(B)$ such that $\mathfrak{W}_i=\mathcal{W}_i \cap \mathcal{M}$ for all $i \in I$. By defining $\bigwedge_{i\in I} \mathfrak{W}_i :=\bigcap_{i\in I}\mathfrak{W}_i$ and the fact that $\bigcap_{i\in I}\mathcal{W}_{i} \in \wide (B)$, it is clear that $\widshad_{\mathcal{M}}(B)$ is a complete meet-semilattice. Since $\text{mod}(B) \cap \mathcal{M}$ is the unique maximal element of $\text{widshad}_\mathcal{M}(B)$, we obtain that $\text{widshad}_\mathcal{M}(B)$ is a complete lattice.


(2) The surjective map $\wide(B) \twoheadrightarrow \widshad_{\mathcal{M}}(B)$ is obviously a poset epimorphism by definition. Furthermore, if $\mathcal{W}, \mathcal{W}'\in \wide(B)$, the image of $\mathcal{W} \land \mathcal{W}'$ is given by $(\mathcal{W} \cap \mathcal{W}') \cap \mathcal{M}$, which is clearly $\mathfrak{W} \land \mathfrak{W}' \in \widshad_{\mathcal{M}}(B)$, where $\mathfrak{W}= \mathcal{W} \cap \mathcal{M}$ and $\mathfrak{W}'=\mathcal{W}'\cap \mathcal{M}$.

The assertion about $(-)\cap \modu(A):\widshad_{\mathcal{M}}(B) \twoheadrightarrow \wide(A)$, is proved in a similar way.\end{proof}

\end{lemma}

We remark that the maps $(-)\cap \mathcal{M}: \wide(B) \twoheadrightarrow \widshad_{\mathcal{M}}(B)$ and $(-)\cap \modu(A):\widshad_{\mathcal{M}}(B) \twoheadrightarrow \wide(A)$ usually fail to be join-semilattice maps.

For the remainder of this section, we let $A=kQ/I$ be a brick gentle algebra, and we let $\mathcal{M}$ be the subcategory of $\modu(\Pi(A))$ defined in Section~\ref{torshad_section}.  Hence, for $\mathcal{W} \in \wide(\Pi(A))$, the associated wide shadow is denoted by $\widetilde{\mathfrak{W}}:= \mathcal{W} \cap \mathcal{M}$ and $\widshad(\Pi(A))$ is the collection of all such subcategories of $\modu(\Pi(A))$ ordered by inclusion.

Before we state the main theorem of this subsection, let us summarize what we have obtained so far in the following diagram, as the main motivation for what follows. 

\begin{center}
\begin{tikzpicture}[scale=1]
\node at (-4,0) {$\tors(\Pi(A))$};
\draw [->>] (-3,0) --(-1.5,0);
\node at (-2.25,0.3) {\tiny $(-)\cap \mathcal{M}$};
---
\node at (-0.25,0) {$\torshad(\Pi(A))$};
\node at (-0.5,0.65) {$\bic(A)$};
\node at (-0.5,0.33) {$\simeqd$};
\draw [->>] (1,0) --(2.5,0);
\node at (1.75,0.3) {\tiny $(-)\cap \modu(A)$};
---
\node at (3.25,0) {$\tors(A)$};
-------
-------
\node at (-4,-2) {$\wide(\Pi(A))$};
\draw [->>] (-3,-2) --(-1.5,-2);
\node at (-2.25,-2.3) {\tiny $(-)\cap \mathcal{M}$};
---
\node at (-0.25,-2) {$\widshad(\Pi(A))$};
\draw [->>] (1,-2) --(2.5,-2);
\node at (1.75,-2.3) {\tiny $(-)\cap \modu(A)$};;
---
\node at (3.25,-2) {$\wide(A)$};
-------
-
---
\draw [dashed, ->] (-0.5,-1.75) --(-0.5,-0.25) ;
\node at (-0.75,-1) {$?$};
-
\draw [dashed, <-] (0,-1.75) --(0,-0.25) ;
\node at (0.25,-1) {$?$};
---
\draw [->] (3,-1.75) --(3,-0.25) ;
\node at (2.7,-1) {\tiny $\mathcal{T}_{(-)}$};
-
\draw [<-] (3.25,-1.75) --(3.25,-0.25) ;
\node at (3.6,-1) {\tiny$\mathcal{W}_{(-)}$};

\end{tikzpicture}
\end{center}

The rightmost vertical maps are the bijections established in \cite{MS}. Furthermore, the horizontal maps are the surjective poset maps described in Lemmas \ref{torshad-epi} and \ref{widshad-epi}. Finally, in Theorem \ref{Bic-Torshad} we proved the isomorphism between $\bic(A)$ and $\torshad(\Pi(A))$. 

We have the following theorem which says that torsion shadows and wide shadows are in bijection. We will prove this theorem by showing that wide shadows are closely linked to the lattice theory of torsion shadows. More specifically, we will show that $\widshad(\Pi(A))$ is isomorphic to the shard intersection order of $\Bic(A)$ in Section~\ref{sec:shard}.

\begin{theorem}\label{Thm_torshad_widshad_bij}
There is a bijection between $\torshad(\Pi(A))$ and $\widshad(\Pi(A))$.
\end{theorem}

We conclude this section with an example of the lattice of wide shadows associated with a brick gentle algebra.

\begin{example}
Assume that $A$ and $\Pi(A)$ are the algebras from Example~\ref{Running example}. In Figure~\ref{widshad_fig_1}, we show the lattice of wide shadows of $A$. Here we describe each wide shadow $\widetilde{\mathfrak{W}}$ by showing the strings defining the string modules in $\widetilde{\mathfrak{W}}$.
\end{example}

\begin{figure}
$$\includegraphics[scale=1]{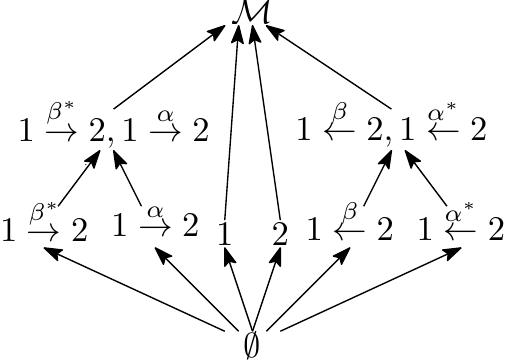}$$
\caption{A lattice of wide shadows.}
\label{widshad_fig_1}
\end{figure}

\section{Canonical join complex for $\Bic(A)$}\label{sec:canonical}

Our next goal is to completely describe the canonical join complex of the lattice of biclosed sets $\Bic(A)$ where $A$ is brick gentle algebra. Our classification of the faces of the canonical join complex will help us to relate the lattice of wide shadows of $A$ to the shard intersection order of $\Bic(A)$.

\begin{theorem}\label{Thm_CJC}
A collection $\{J(w^1_{\mathcal{D}^1}), \ldots, J(w^k_{\mathcal{D}^k})\} \subset \JI(\Bic(A))$ is a face of the canonical join complex $\Delta^{CJ}(\Bic(A))$ if and only if labels $w^i_{\mathcal{D}^i}$ and $w^j_{\mathcal{D}^j}$ satisfy the following:
\begin{itemize}
\item[1)] strings $w^i$ and $w^j$ are distinct,
\item[2)] neither $w^i$ nor $w^j$ is expressible as a concatenation of at least two strings in $J(w^i_{\mathcal{D}^i})\cup J(w^j_{\mathcal{D}^j})$, and
\item[3)] neither $J(w^i_{\mathcal{D}^i}) \le J(w^{j}_{\mathcal{D}^j})$ nor $J(w^{j}_{\mathcal{D}^j}) \le J(w^i_{\mathcal{D}^i})$
\end{itemize}
for any distinct $i,j \in \{1,\ldots, k\}$.
\end{theorem}

Before presenting the proof of Theorem~\ref{Thm_CJC}, we mention the following corollary that we will use when we discuss the shard intersection order of $\Bic(A)$.

\begin{corollary}\label{Cor_bic_hom_orthog}
Let $B \in \Bic(A)$, and let $\{J(w^1_{\mathcal{D}^1}), \ldots, J(w^k_{\mathcal{D}^k})\}$ be the canonical joinands in its canonical join representation. Then $$\Hom_{\Pi(A)}(M(\text{str}(w^i_{\mathcal{D}^i})), M(\text{str}(w^j_{\mathcal{D}^j}))) = 0$$ for any $i \neq j$. \end{corollary}

\begin{proof} By Theorem~\ref{Thm_CJC}, we know that $M(\text{str}(w^i_{\mathcal{D}^i}))\not \cong M(\text{str}(w^j_{\mathcal{D}^j}))$ for any $i \neq j$.  

Let $f \in \Hom_{\Pi(A)}(M(\text{str}(w^i_{\mathcal{D}^i})), M(\text{str}(w^j_{\mathcal{D}^j})))$. Since $\text{im}(f)$ is a quotient of $M(\text{str}(w^i_{\mathcal{D}^i}))$, it is isomorphic to a (possibly empty) direct sum of string modules defined by substrings of $\text{str}(w^i_{\mathcal{D}^i})$ no two of which contain a common vertex. Similarly, since $\text{im}(f)$ is a submodule of $M(\text{str}(w^j_{\mathcal{D}^j}))$, the summands of $\text{im}(f)$ must be string modules defined by substrings of $\text{str}(w^j_{\mathcal{D}^j})$ no two of which contain a common vertex. 

Now, let $M(\text{str}(w_\mathcal{D}))$ be a summand of $\text{im}(f)$. Since $M(\text{str}(w_\mathcal{D}))$ is a submodule of $M(\text{str}(w^j_{\mathcal{D}^j}))$, we know that there does not exist any $u \in \mathcal{D}^j$ such that $w \in \{u\} \sqcup S(u, \mathcal{D}^j)$. Similarly, since $M(\text{str}(w_\mathcal{D}))$ is a quotient of $M(\text{str}(w^i_{\mathcal{D}^i}))$, there exists $v \in \mathcal{D}^i$ such that $w \in \{v\} \sqcup S(v, \mathcal{D}^i)$. If $w$ is not a split of $w^j$, there exists $u^\prime, u^{\prime\prime} \in \mathcal{D}^j$ and $\alpha, \beta \in {Q}_1$ such that $w^j = u^{\prime}\alpha^{\pm 1} w \beta^{\pm 1} u^{\prime\prime}$. We obtain a similar equation if $w$ is a split of $w^j$. In each case, by Theorem~\ref{Thm_CJC}, this contradicts that $J(w^i_{\mathcal{D}^i})$ and $J(w^j_{\mathcal{D}^j})$ are canonical joinands of the canonical join representation of $B$.\end{proof}


\begin{example}
Assume that $A$ is the algebra from Example~\ref{Running example}. In Figure~\ref{cjc_figure}, we show the canonical join complex $\Delta^{CJ}(\Bic(A)).$
\end{example}

\begin{figure}
$$\includegraphics[scale=1]{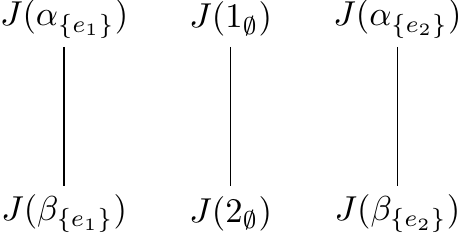}$$
\caption{A canonical join complex. The vertices of this complex are elements of $\JI(\Bic(A))$ and a collection of join-irreducibles are incident in this complex if and only if the join of all of those join-irreducibles is a canonical join representation of some biclosed set.}
\label{cjc_figure}
\end{figure}

\begin{proof}[Proof of Theorem~\ref{Thm_CJC}]
Let $\{J(w^1_{\mathcal{D}^1}), \ldots, J(w^k_{\mathcal{D}^k})\} \subset \text{JI}(\text{Bic}(A))$ where there exists distinct $i, j \in \{1,\ldots, k\}$ such that $w^i_{\mathcal{D}^i}$ and  $w^j_{\mathcal{D}^j}$ do not satisfy all of the stated properties. To prove that $\{J(w^1_{\mathcal{D}^1}), \ldots, J(w^k_{\mathcal{D}^k})\}$ is not a face of $\Delta^{CJ}(\Bic(A))$, it is enough to show that $J(w^i_{\mathcal{D}^i})\vee J(w^j_{\mathcal{D}^j})$ is not a canonical join representation; cf. Lemma~\ref{lem:flag}.

If $w^i = w^j$, then by Lemma~\ref{nosplit} there does not exist $B \in \text{Bic}(A)$ such that $w^i_{\mathcal{D}}, w^j_{\mathcal{D}^\prime} \in {\lambda}_\downarrow(B)$ for any subsets $\mathcal{D}, \mathcal{D}^\prime \subset \Str(A)$. Now by Lemma~\ref{Lemma_cjr_cmr}, we have that $J(w^i_{\mathcal{D}^i})\vee J(w^j_{\mathcal{D}^j})$ is not a canonical join representation.

Next, suppose that $w^i$ or $w^j$ may be expressed as a concatenation of at least two strings in $J(w^{i}_{\mathcal{D}^i}) \cup J(w^j_{\mathcal{D}^j})$. Then Lemma~\ref{lemma_comp_not_canonical} implies that $J(w^i_{\mathcal{D}^i})\vee J(w^j_{\mathcal{D}^j})$ is not a canonical join representation.

Lastly, suppose that, without loss of generality, $J(w^i_{\mathcal{D}^i}) \le J(w^{j}_{\mathcal{D}^j})$. This implies that $J(w^i_{\mathcal{D}^i}) \vee J(w^{j}_{\mathcal{D}^j}) = J(w^{j}_{\mathcal{D}^j})$ and so the expression $J(w^i_{\mathcal{D}^i}) \vee J(w^{j}_{\mathcal{D}^j})$ is not an irredundant join representation.

Conversely, suppose $\{J(w^1_{\mathcal{D}^1}), \ldots, J(w^k_{\mathcal{D}^k})\} \subset \text{JI}(\text{Bic}(A))$ and that any pair of distinct labels $w^i_{\mathcal{D}^i}$ and $w^j_{\mathcal{D}^j}$  satisfy all of the stated properties. Then Lemma~\ref{Lemma_is_canonical} implies that $J(w^i_{\mathcal{D}^i})\vee J(w^j_{\mathcal{D}^j})$ is a canonical join representation for any distinct $i,j \in \{1,\ldots, k\}$. Now, using Lemma~\ref{lem:flag}, we have that $\bigvee_{i = 1}^k J(w^i_{\mathcal{D}^i})$ is a canonical join representation. Thus $\{J(w^1_{\mathcal{D}^1}), \ldots, J(w^k_{\mathcal{D}^k})\} \subset \text{JI}(\text{Bic}(A))$ is a face of $\Delta^{CJ}(\text{Bic}(A))$.
\end{proof}

The remainder of this section is dedicated to proving the lemmas cited in the proof of Theorem~\ref{Thm_CJC}.

\begin{lemma}\label{nosplit}
Given $B \in \Bic(A)$ and distinct covering relations $(B_1,B), (B_2,B) \in \Cov(\Bic(A))$, let $w^1_{\mathcal{D}^1} = \widetilde{\lambda}(B_1,B)$ and $w^2_{\mathcal{D}^2} = \widetilde{\lambda}(B_2,B)$. The string $w^1$ is not a split of $w^2$, $w^2$ is not a split of $w^1$, and $w^1 \neq w^2$.
\end{lemma}

\begin{proof}
Since $B_1 \neq B_2$, it is clear that $w^1\neq w^2$.

To complete the proof, it is enough to show that $w^1$ is not a split of $w^2$. Suppose that $w^1$ is a split of $w^2$. By Lemma~\ref{Lemma_covers_in_bic}, we have that $B_1=B\backslash\{w^1\}$ and $B_2=B\backslash\{w^2\}$. Now let $w^\prime \in \Str(A)$ denote the string satisfying $w^1\alpha^{\pm 1} w^\prime = w^2$ where $\alpha \in Q_1$. Observe that since $w^1 \not \in B_1$, we have $w^\prime \in B_1$. This implies that $w^\prime \in B$ and so $w^\prime \in B_2$. However, this means that $w^1, w^\prime \in B_2$, but $w^2 = w^1\alpha^{\pm 1} w^\prime \not \in B_2$, which contradicts that $B_2$ is closed.
\end{proof}

\begin{lemma}\label{Lemma_join_irr_containment}
Let $w_\mathcal{D} \in \mathcal{S}$, let $v \in \{w\}\sqcup \mathcal{D} \sqcup \bigcup_{u \in \mathcal{D}} S(u)$, and let $\mathcal{D}(v):= \{w^\prime \in J(w_\mathcal{D}): \text{$w^\prime$ is a split of $v$}\}$. Then $J(v_{\mathcal{D}(v)}) \in \JI(\Bic(A))$ and $J(v_{\mathcal{D}(v)}) \le J(w_\mathcal{D})$. 
\end{lemma}
\begin{proof}
By Lemma~\ref{lemma_J_biclosed}, given $v \in \{w\}\sqcup \mathcal{D} \sqcup \bigcup_{u\in \mathcal{D}} S(u)$ and a break $\{w^1,w^2\}$ of $v$ exactly one of these splits belongs to $\{w\}\sqcup \mathcal{D}\sqcup \bigcup_{u \in \mathcal{D}} S(u)$. This implies that $v_{\mathcal{D}(v)} \in \mathcal{S}$. By Proposition~\ref{Prop_join_irr_des}, $J(v_{\mathcal{D}(v)})$ is a join-irreducible biclosed set.

Next, we show that $J(v_{\mathcal{D}(v)}) \le J(w_\mathcal{D})$. It follows from the proof of \cite[Lemma 4.3]{cliftondillerygarver} that $\{v\} \sqcup \mathcal{D}(v)\sqcup \bigcup_{v^\prime \in \mathcal{D}(v)}S(v^\prime)$ is contained in $\{w\} \sqcup \mathcal{D} \sqcup \bigcup_{u \in \mathcal{D}} S(u)$. Therefore the former is contained in $J(w_\mathcal{D})$. Since $J(w_\mathcal{D})$ is closed, we obtain that $J(v_{\mathcal{D}(v)}) \le J(w_\mathcal{D}).$\end{proof}

\begin{remark}
Lemma~\ref{Lemma_join_irr_containment} is false if $u \in J(w_\mathcal{D})\backslash\left(\{w\}\sqcup \mathcal{D} \sqcup \bigcup_{u \in \mathcal{D}}S(u)\right)$. This is because such a string $u$ must have a break $\{u^1, u^2\}$ where $u^1, u^2 \in \{w\}\sqcup \mathcal{D} \sqcup \bigcup_{u \in \mathcal{D}}S(u).$ Therefore, the expression $u_{\mathcal{D}(u)}$ is not an element of $\mathcal{S}$.
\end{remark}

\begin{lemma}\label{lemma_comp_not_canonical}
Given $w_\mathcal{D}, w^\prime_{\mathcal{D}^\prime} \in \mathcal{S}$. Assume that there exists $u^1, \ldots, u^k \in J(w_\mathcal{D})\cup J(w^\prime_{\mathcal{D}^\prime})$ with $k \ge 2$ such that $w = u^1\alpha_1^{\pm 1}u^2 \cdots u^{k-1}\alpha_{k-1}^{\pm 1}u^k$ for some $\alpha_1, \ldots, \alpha_{k-1} \in Q_1$. Then $J({w}_\mathcal{D}) \vee J(w^\prime_{\mathcal{D}^\prime})$ is not a canonical join representation.
\end{lemma}
\begin{proof}
We can assume that $w \neq w^\prime$, $w^\prime$ is not a split of $w$, and $w$ is not a split of $w^\prime$, otherwise we obtain the desired result from Lemma~\ref{nosplit} and Lemma~\ref{Lemma_cjr_cmr}.

Assume that there exists $u^1, \ldots, u^k \in J(w_\mathcal{D})\cup J(w^\prime_{\mathcal{D}^\prime})$ such that $$w = u^1\alpha_1^{\pm 1}u^2 \cdots u^{k-1}\alpha_{k-1}^{\pm 1}u^k$$ with $k \ge 2$ for some $\alpha_1, \ldots, \alpha_{k-1} \in Q_1$. Observe that $w$ has such an expression where the following hold: 
\begin{itemize}
\item[i)] there exists $i \in \{1,\ldots, k\}$ such that $u^i \in J(w_\mathcal{D})\backslash J(w^\prime_{\mathcal{D}^\prime})$, and 
\item[ii)]$u^i$ cannot be expressed as the concatenation of at least two elements of $J(w_\mathcal{D})\cup J(w^\prime_{\mathcal{D}^\prime})$ for any $i \in \{1,\ldots, k\}$.
\end{itemize}
From the set of all such expressions for $w$, let $u^1, \ldots, u^k \in J(w_\mathcal{D}) \cup J(w^\prime_{\mathcal{D}^\prime})$ have the following properties: 
\begin{itemize}
\item string $u^1$ satisfies ii) and is a maximal length string satisfying ii);  
\item assuming, by induction, that $u^1, \ldots, u^{i}$ satisfy ii) and are maximal length strings satisfying ii), string $u^{i+1}$ satisfies ii) and is a maximal length string satisfying ii).
\end{itemize}
Now let $u^{i_1},\ldots, u^{i_\ell}$ denote the strings in this expression for $w$ that belong to $J(w_\mathcal{D})\backslash J(w^\prime_{\mathcal{D}^\prime}).$ We show that $\left(\bigvee_{j=1}^\ell J(u^{i_j}_{\mathcal{D}(u^{i_j})})\right) \vee J(w^\prime_{\mathcal{D}^\prime})$ is a refinement of $J({w}_\mathcal{D}) \vee J(w^\prime_{\mathcal{D}^\prime})$ by showing that the two expressions are equal.

First, we know that $u^{i_j} \in \{w\}\sqcup \mathcal{D} \sqcup \bigcup_{u \in \mathcal{D}}S(u)$ for all $j \in \{1, \ldots, \ell\}$ because each $u^{i_j}$ is a substring of $w$. Therefore, by Lemma~\ref{Lemma_join_irr_containment}, $J(u^{i_j}_{\mathcal{D}^{i_j}}) \le J(w_{\mathcal{D}})$ for all $j \in \{1, \ldots, \ell\}$. This implies that $\left(\bigvee_{j=1}^\ell J(u^{i_j}_{\mathcal{D}(u^{i_j})})\right) \vee J(w^\prime_{\mathcal{D}^\prime}) \le J({w}_\mathcal{D}) \vee J(w^\prime_{\mathcal{D}^\prime})$.

Next, we prove the opposite containment. Note that any element of $J(w_\mathcal{D}) \vee J(w^\prime_{\mathcal{D}^\prime})$ is a concatenation of elements of $J(w^\prime_{\mathcal{D}^\prime})$ and substrings of $w$ that are contained in $J(w_\mathcal{D})\vee J(w^\prime_{\mathcal{D}^\prime})$. As $\left(\bigvee_{j=1}^\ell J(u^{i_j}_{\mathcal{D}(u^{i_j})})\right) \vee J(w^\prime_{\mathcal{D}^\prime})$ is closed, it is enough to prove that any substring of $w$ contained in $J(w_\mathcal{D})\vee J(w^\prime_{\mathcal{D}^\prime})$ belongs to $\left(\bigvee_{j=1}^\ell J(u^{i_j}_{\mathcal{D}(u^{i_j})})\right) \vee J(w^\prime_{\mathcal{D}^\prime})$. If $v \in J(w_\mathcal{D})\vee J(w^\prime_{\mathcal{D}^\prime})$ is a substring of $u^i$ for some $i = 1, \ldots, k$, we have that $v$ is a concatenation of strings in $J(u^{i}_{\mathcal{D}(u^i)}) \cup J(w^\prime_{\mathcal{D}^\prime})$. Thus, $v$ belongs to $\left(\bigvee_{j=1}^\ell J(u^{i_j}_{\mathcal{D}(u^{i_j})})\right) \vee J(w^\prime_{\mathcal{D}^\prime}).$ This means we must show that if $v \in J(w_\mathcal{D})\vee J(w^\prime_{\mathcal{D}^\prime})$ and $v$ is a substring of $w$, then $v \in \left(\bigvee_{j=1}^\ell J(u^{i_j}_{\mathcal{D}(u^{i_j})})\right) \vee J(w^\prime_{\mathcal{D}^\prime})$ when one of the following cases holds:
\begin{itemize}
\item[\textit{1)}] $v = u^1 \alpha_1^{\pm 1} u^2  \cdots  u^{s-1}\alpha_{s-1}^{\pm 1} u_s^\prime$ for some $u_s^\prime \in J(w_\mathcal{D})\cup J(w^\prime_{\mathcal{D}^\prime})$,
\item[\textit{2)}] $v = u_r^\prime \alpha_{r}^{\pm 1} u^{r+1}  \cdots  u^{k-1}\alpha_{k-1}^{\pm 1} u^k$ for some $u_r^\prime \in J(w_\mathcal{D})\cup J(w^\prime_{\mathcal{D}^\prime})$, or 
\item[\textit{3)}] $v = u_r^\prime \alpha_{r}^{\pm 1} u^{r+1}  \cdots  u^{s-1}\alpha_{s-1}^{\pm 1} u_s^\prime$ for some $u_r^\prime, u_s^\prime \in J(s_\mathcal{D})\cup J(s^\prime_{\mathcal{D}^\prime})$.
\end{itemize}
We verify \textit{Case 2)}, and the proof of \textit{Case 1)} and \textit{3)} is similar to that of \textit{Case 2)}.

\textit{Case 2):} We show that $u_r^\prime \in J(u^r_{\mathcal{D}(u^r)}).$ Note that $J(u^r_{\mathcal{D}(u^r)})$ is well-defined and $J(u^r_{\mathcal{D}(u^r)}) \le J(w^\prime_{\mathcal{D}^\prime})$ by our proof of the first containment. Suppose $u_r^\prime \not \in J(u^r_{\mathcal{D}(u^r)})$. Since $u_r^\prime$ is a split of $u^r$, we may write $u_r^{\prime\prime}\alpha^{\pm 1} u_r^\prime = u^r$ for some $u^{\prime\prime}_r \in \Str(A)$ and some $\alpha \in Q_1$. As $u_r^\prime$ does not belong to $J(u^r_{\mathcal{D}(u^r)})$, we know that $u_r^{\prime\prime} \in \mathcal{D}(u^r)$. We also know $u_r^\prime \in J(w_\mathcal{D})\cup J(w^\prime_{\mathcal{D}^\prime})$ and so the expression $u^r = u_r^{\prime\prime}\alpha^{\pm 1} u_r^\prime$ contradicts our choice of $u^r$.
\end{proof}




\begin{lemma}\label{Lemma_is_canonical}
Let ${w}_\mathcal{D}, w^\prime_{\mathcal{D}^\prime} \in \mathcal{S}$ be labels with the following properties:
\begin{itemize}
\item[1)] strings $w$ and $w^\prime$ are distinct,
\item[2)] neither $w$ nor $w^\prime$ is expressible as a concatenation of at least two strings in $J(w_\mathcal{D})\cup J(w^\prime_{\mathcal{D}^\prime})$, and
\item[3)] neither $J(w_{\mathcal{D}}) \le J(w^{\prime}_{\mathcal{D}^\prime})$ nor $J(w^{\prime}_{\mathcal{D}^\prime}) \le J(w_{\mathcal{D}})$.
\end{itemize}Then $J({w}_\mathcal{D}) \vee J(w^\prime_{\mathcal{D}^\prime})$ is a canonical join representation.
\end{lemma}

\begin{proof}By the stated properties satisfied by $w_\mathcal{D}$ and $w^\prime_{\mathcal{D}^\prime},$ there exist strings $u \in J(w_{\mathcal{D}})\backslash J(w^\prime_{\mathcal{D}^\prime})$ and $u^\prime \in J(w^\prime_{\mathcal{D}^\prime})\backslash J(w_{\mathcal{D}})$. This implies that $J(w_\mathcal{D}) < J(w_\mathcal{D})\vee J(w^\prime_{\mathcal{D}^\prime})$ and $J(w^\prime_{\mathcal{D}^\prime}) < J(w_\mathcal{D})\vee J(w^\prime_{\mathcal{D}^\prime}).$ Therefore, the join representation $J(w_\mathcal{D})\vee J(w^\prime_{\mathcal{D}^\prime})$ is irredundant.

Next, suppose that $J(w_\mathcal{D})\vee J(w^\prime_{\mathcal{D}^\prime}) = \bigvee_{i = 1}^k J(u^i_{\mathcal{D}^i})$ where the latter is irredundant. We will show that $J(w_\mathcal{D}) \le J(u^i_{\mathcal{D}^i})$ for some $i = 1, \ldots, k$, and one uses the same strategy to prove that $J(w^\prime_{\mathcal{D}^\prime}) \le J(u^j_{\mathcal{D}^j})$ for some $j = 1, \ldots, k$.

Since $w \in \bigvee_{i = 1}^k J(u^i_{\mathcal{D}^i})$, there exist $u_{i_j} \in J(u^{i_j}_{\mathcal{D}^{i_j}})$ with $j = 1, \ldots, \ell$ such that $w = u_{i_1}\alpha_{i_1}^{\pm 1}u_{i_2} \cdots u_{i_{\ell-1}}\alpha_{i_{\ell-1}}^{\pm 1}u_{i_\ell}$ for some $\alpha_{i_{j}} \in Q_1$ with $j \in \{1, \ldots, \ell -1\}$. By the fact that $J(w_\mathcal{D})\vee J(w^\prime_{\mathcal{D}^\prime}) = \bigvee_{i = 1}^k J(u^i_{\mathcal{D}^i})$, we can assume $$u_{i_j} \in \left(\{w\}\sqcup \mathcal{D}\sqcup \bigcup_{u \in \mathcal{D}}S(u)\right)\bigcup \left(\{w^\prime\} \sqcup \mathcal{D}^\prime \sqcup \bigcup_{u^\prime \in \mathcal{D}^\prime}S(u^\prime)\right)$$ for all $j = 1, \ldots, \ell$. As $w$ is not expressible as a concatenation of at least two strings from $J(w_\mathcal{D})\cup J(w^\prime_{\mathcal{D}^\prime})$, this implies that $\ell = 1$ and so $w \in J(u^i_{\mathcal{D}^i})$ for some $i = 1, \ldots, k$.


Now let $u \in \mathcal{D}$. We can write $w = u \alpha^{\pm 1} v$ for some $v \in \Str(A)$ and some $\alpha \in Q_1$. Suppose $u \not \in J(u^i_{\mathcal{D}^i})$.  Since $J(u^i_{\mathcal{D}^i})$ is biclosed and $w \in J(u^i_{\mathcal{D}^i})$, we know $v \in J(u^i_{\mathcal{D}^i}).$ However, by the fact that $J(w_\mathcal{D})\vee J(w^\prime_{\mathcal{D}^\prime}) = \bigvee_{i = 1}^k J(u^i_{\mathcal{D}^i})$, the equation $w = u \alpha^{\pm 1} v$ contradicts that $w$ is not expressible as a concatenation of at least two strings from $J(w_\mathcal{D})\cup J(w^\prime_{\mathcal{D}^\prime})$. This implies that $\mathcal{D} \subseteq J(u^i_{\mathcal{D}^i})$ and no other splits of $w$ belong to $J(u^i_{\mathcal{D}^i})$. Thus $u \in J(u^i_{\mathcal{D}^i})$ so $\mathcal{D} = \{u \in J(u^i_{\mathcal{D}^i}): \ u \text{ is a split of } w\}.$

We now conclude from Lemma~\ref{Lemma_join_irr_containment} that $J(w_\mathcal{D}) \le J(u^i_{\mathcal{D}^i})$ so $J({w}_\mathcal{D}) \vee J(w^\prime_{\mathcal{D}^\prime})$ is a canonical join representation.
\end{proof}

\section{The shard intersection order of $\Bic(A)$}\label{sec:shard}

We now relate the shard intersection order $\Psi(\Bic(A))$ to the lattice of wide shadows $\widshad(\Pi(A))$.

\begin{theorem}\label{thm_psi_shad_isom}
If $A$ is a brick gentle algebra, there is a poset isomorphism given by $$\begin{array}{rcl}\Psi(\Bic(A)) &\stackrel{\vartheta}{\longrightarrow} & \widshad(\Pi(A))\\ \psi(B) & \longmapsto & \add(\bigoplus_{w_\mathcal{D} \in \psi(B)} M(\str(w_{\mathcal{D}}))).\end{array}$$
\end{theorem}

We prove this by establishing several lemmas. 

\begin{lemma}\label{lemma_Psi_to_shad}
One has the following order-preserving map 
$$\begin{array}{rcl}\Psi(\Bic(A)) & \stackrel{\vartheta}{\longrightarrow}& \widshad(\Pi(A))\\ \psi(B) & \longmapsto & \add(\bigoplus_{w_\mathcal{D} \in \psi(B)} M(\str(w_{\mathcal{D}}))).\end{array}$$
\end{lemma}
\begin{proof}
Let $B \in \Bic(A)$, and let $\{J(w^1_{\mathcal{D}^1}), \ldots, J(w^k_{\mathcal{D}^k})\}$ be the canonical joinands in its canonical join representation. By Lemma~\ref{bricks_in_Pi}, Corollary~\ref{Cor_bic_hom_orthog}, and \cite[Theorem]{Ringel}, the extension closure of $\{M(\text{str}(w^i_{\mathcal{D}^i}))\}_{i = 1}^k$, denoted $\mathcal{W}$, is a wide subcategory of $\text{mod}(\Pi(A))$. By referring to Figure~\ref{polygons_fig}, for any $w_\mathcal{D} \in \psi(B)$ its corresponding string $\str(w_\mathcal{D})$ is a concatenation of some of the strings in $\{\text{str}(w^i_{\mathcal{D}^i})\}_{i = 1}^k$. Thus $M(\str(w_\mathcal{D})) \in \mathcal{W}$. By Lemma~\ref{Lemma_labels_to_Pi_strings}, given $w_\mathcal{D} \in \psi(B)$, we see that $M(\text{str}(w_\mathcal{D})) \in \mathcal{M}$. Thus $\text{add}(\bigoplus_{w_\mathcal{D}\in \psi(B)} M(\text{str}(w_{\mathcal{D}}))) \subset \mathcal{W}\cap \mathcal{M}.$ %

Conversely, suppose that $M(\text{str}(w_\mathcal{D})) \in \mathcal{W}\cap \mathcal{M}.$ Since $M(\text{str}(w_\mathcal{D})) \in \mathcal{W}$, $M(\text{str}(w_\mathcal{D}))$ has a filtration $0 = X_0 \subset X_1 \subset \cdots \subset X_m = M(\text{str}(w_\mathcal{D}))$ where for each $i = 1, 2, \ldots, m$ one has $X_i/X_{i-1} = M(\text{str}(w^j_{\mathcal{D}^j}))$ for some $j = 1, \ldots, k.$ As $M(\text{str}(w_\mathcal{D})) \in \mathcal{M}$, no two quotients $X_{i}/X_{i-1}$ and $X_{i^\prime}/X_{i^\prime-1}$ with $i \neq i^\prime$ are isomorphic. Thus $w$ is a concatenation of a subset of the strings $w^1, \ldots, w^k.$ Now by referring to Figure~\ref{polygons_fig}, we see that $w_\mathcal{D} \in \psi(B).$

It is obvious that this map is order-preserving.
\end{proof}

Next, by Lemma~\ref{Lemma_labels_to_Pi_strings}, there is a map $\text{widshad}(\Pi(A))\to 2^\mathcal{S}$ given by sending a given wide shadow $\widetilde{\mathfrak{W}}$ to the set of labels defining the string modules in $\widetilde{\mathfrak{W}}$. Let $W \subset \mathcal{S}$ denote the image of $\widetilde{\mathfrak{W}}$ under this map.

\begin{lemma}\label{lemma_shad_simples}
Given any nonzero wide shadow $\widetilde{\mathfrak{W}} \in \widshad(\Pi(A))$, there exists a nonempty subset $\Sim(\widetilde{\mathfrak{W}}) \subset \widetilde{\mathfrak{W}}$ consisting of the elements of $\widetilde{\mathfrak{W}}$ of the form $M(\str(w_\mathcal{D}))$ where $w$ appears in exactly one label in $W$. We let $\Sim(W) \subset W$ denote the set of labels defining the modules in $\Sim(\widetilde{\mathfrak{W}})$.
\end{lemma}
\begin{proof}
Suppose that there does not exist a string $w$ appearing in exactly one label in $W$. This means there is no label of the form ${e_i}_\emptyset$ in $W$ for any $i \in Q_0$. Therefore, let $w_\mathcal{D}, w_{\mathcal{D}^\prime} \in W$ where $w$ is a string of minimal length. Write $\str(w_\mathcal{D}) = \overline{\gamma_d}^{\epsilon_d}\cdots \overline{\gamma_1}^{\epsilon_1}$ and $\str(w_{\mathcal{D}^\prime}) = {\overline{\gamma_d}}^{\epsilon_d^\prime}\cdots\overline{\gamma_1}^{\epsilon_1^\prime}$. Observe that there exists $u \in \mathcal{D}$ and $u^\prime \in \mathcal{D}^\prime$ such that $w = u\gamma_i^{\pm 1} u^\prime$ for some $\gamma_i \in {Q}_1$. Let $i \in \{1,\ldots, d\}$ be maximal such that there exists $u \in \mathcal{D}$ and $u^\prime \in \mathcal{D}^\prime$ such that $w = u\gamma_i^{\pm 1} u^\prime$.


From the definition of the map $\str(-)$, this implies that there is a homomorphism $f: M(\text{str}{(w_\mathcal{D})}) \to M(\text{str}(w_{\mathcal{D}^\prime}))$ satisfying $\text{im}(f) = M(\text{str}(u_{\mathcal{D}(u)}))$. Thus $\text{coker}(f) = M(\text{str}(u^\prime_{\mathcal{D}^\prime(u^\prime)}))$. 

Now write $\widetilde{\mathfrak{W}} = \mathcal{W} \cap\mathcal{M}$. Since $\mathcal{W}$ is abelian, $M(\text{str}(u^\prime_{\mathcal{D}^\prime(u^\prime)})) \in \mathcal{W}$. We also know that $M(\text{str}(u^\prime_{\mathcal{D}^\prime(u^\prime)})) \in \mathcal{M}$. However, this contradicts the minimality of $w$. We obtain the desired result.
\end{proof}



\begin{remark}\label{splitting_remark}
The proof of Lemma~\ref{lemma_shad_simples} implies the following useful fact. Given $M(\str(w_\mathcal{D}))$ and $M(\str(w_{\mathcal{D}^\prime}))$ where $\mathcal{D} \neq \mathcal{D}^{\prime}$, then we can write $w = u\alpha^{\pm 1}u^\prime$ with $u \in \mathcal{D}$ and $u^\prime \in \mathcal{D}^\prime$ such that there is a homomorphism $f: M(\str(w_\mathcal{D})) \to M(\str(w_{\mathcal{D}^\prime}))$ with $\text{im}(f) = M(\str(u_{\mathcal{D}(u)}))$ and $\text{coker}(f) = M(\str(u^\prime_{\mathcal{D}^\prime(u^\prime)}))$. 
\end{remark}

\begin{lemma}\label{lemma_psiB_in_W}
The set $\{J(w_\mathcal{D}): \ w_{\mathcal{D}} \in \Sim({W})\}$ is a face of $\Delta^{CJ}(\Bic(A))$ for any $\widetilde{\mathfrak{W}} \in {\widshad}(\Pi(A))$. By defining $B := \bigvee_{w_{\mathcal{D}} \in \Sim(W)} J(w_\mathcal{D})$, we have  $\psi(B) \subset W.$
\end{lemma}
\begin{proof}
We show that the elements of $\text{Sim}(W)$ satisfy Theorem~\ref{Thm_CJC}. Let $w_\mathcal{D}, w^\prime_{\mathcal{D}^\prime} \in \text{Sim}(W)$. By definition, $w \neq w^\prime$, which verifies 1).

Next, suppose without loss of generality that $w^\prime = u^1\alpha_1^{\pm 1}u^2\cdots u^{\ell-1}\alpha_{\ell-1}^{\pm 1}u^\ell$ for some $u^1,\ldots, u^\ell \in J(w_\mathcal{D})\cup J(w^\prime_{\mathcal{D}^\prime})$ and some $\alpha_1,\ldots, \alpha_{\ell-1} \in Q_1.$ It follows from the definition of $J(w^\prime_{\mathcal{D}^\prime})$ that $w^\prime$ may not be expressed as a concatenation of strings in $J(w^\prime_{\mathcal{D}^\prime})$. Thus, there exists $i \in \{1, \ldots, \ell\}$ such that $u^i \in J(w_\mathcal{D})\backslash J(w^\prime_{\mathcal{D}^\prime})$. This implies that we have a surjection $\pi: M(\str(w_\mathcal{D})) \twoheadrightarrow M(\str(u^i_{\mathcal{D}(u^i)}))$. 

Now, define $\mathcal{D}^{u^i}$ to be the set of splits of $u^i$ that realizes $\text{str}(u^i_{\mathcal{D}^{u^i}})$ as a proper substring of $\text{str}(w^\prime_{\mathcal{D}^\prime})$. By the definition of $\str(-)$, this implies that there is an inclusion $\iota^\prime: M(\text{str}(u^i_{\mathcal{D}^{u^i}})) \hookrightarrow M(\str(w^\prime_{\mathcal{D}^\prime}))$. From, Remark~\ref{splitting_remark} implies that we can write $u^i = u\alpha^{\pm 1} u^\prime$ for some $u \in \mathcal{D}(u^i)$ and $u^\prime \in \mathcal{D}^{u^i}$ such that there this a homomorphism $f: M(\str(u^i_{\mathcal{D}(u^i)}))\to M(\str(u^i_{\mathcal{D}^{u^i}}))$ whose image (resp., cokernel) is a string module defined by a proper substring of $\str(u^i_{\mathcal{D}(u^i)})$ (resp., $\str(u^i_{\mathcal{D}^{u^i}})$).

Consequently, the cokernel of the map $\iota^\prime \circ f \circ\pi:M(\str(w_\mathcal{D})) \to M(\str(w^\prime_{\mathcal{D}^\prime}))$ is a direct sum of one or two string modules each of which is defined by a split of $\str(w^\prime_{\mathcal{D}^\prime})$, and this cokernel must belong to $\widetilde{\mathfrak{W}}$. Let $\str(v^\prime_{\mathcal{D}^\prime(v^\prime)})$ be one such string defining a summand of this cokernel. Notice that the kernel of the surjection $\pi^\prime: M(\str(w^\prime_{\mathcal{D}^\prime})) \twoheadrightarrow M(\str(v^\prime_{\mathcal{D}^\prime(v^\prime)}))$ is a string module $M(\str(u^\prime_{\mathcal{D}^{u^\prime}}))$ where $\mathcal{D}^{u^\prime}$ is the set of splits of $u^\prime$ that realizes $\str(u^\prime_{\mathcal{D}^{u^\prime}})$ as a proper substring of $\str(w^\prime_{\mathcal{D}^\prime})$. As it is the kernel of $\pi^\prime$, we have $M(\str(u^\prime_{\mathcal{D}^{u^\prime}})) \in \widetilde{\mathfrak{W}}$. 

Next, $\mathcal{D}^\prime_* = \mathcal{D}^\prime\backslash\{v^\prime, u^\prime\}$ we have the following extensions: $$0 \to M(\text{str}(u^\prime_{\mathcal{D}^{u^\prime}})) \to M(\text{str}(w^\prime_{(\mathcal{D}^\prime\backslash\{v^\prime, u^\prime\})\sqcup\{v^\prime\}})) \to M(\text{str}(v^\prime_{\mathcal{D}^\prime({v^\prime})})) \to 0$$
and
$$0 \to M(\text{str}(v^\prime_{\mathcal{D}^\prime({v^\prime})})) \to M(\text{str}(w^\prime_{(\mathcal{D}^\prime\backslash\{v^\prime, u^\prime\})\sqcup\{u^\prime\}})) \to M(\text{str}(u^\prime_{\mathcal{D}^{u^\prime}})) \to 0.$$ Thus both of the middle terms of these extensions belong to $\widetilde{\mathfrak{W}}$. We obtain that both of the labels $w^\prime_{(\mathcal{D}^\prime\backslash\{v^\prime, u^\prime\})\sqcup\{v^\prime\}}$ and $w^\prime_{(\mathcal{D}^\prime\backslash\{v^\prime, u^\prime\})\sqcup\{u^\prime\}}$ belong to $W$, contradicting that $w^\prime$ appears in a single label in $W$. This verifies 2). 

Now suppose that $J(w^\prime_{\mathcal{D}^\prime}) \le J(w_\mathcal{D})$. If $w^\prime \in \{w\} \sqcup \mathcal{D}\sqcup \bigcup_{u\in\mathcal{D}}S(u,\mathcal{D})$, then $\mathcal{D}^\prime = \mathcal{D}(w^\prime).$ In particular, there is a surjection $M(\text{str}(w_\mathcal{D})) \twoheadrightarrow M(\text{str}(w^\prime_{\mathcal{D}(w^\prime)}))$. The kernel of this map is $M(\text{str}(u_{\mathcal{D}^u}))\oplus M(\text{str}(v_{\mathcal{D}^v}))$ for some strings $\text{str}(u_{\mathcal{D}^u})$ and $\text{str}(v_{\mathcal{D}^v})$ with the caveat that at most one of these strings may be the empty string. Here $\mathcal{D}^u$ (resp., $\mathcal{D}^{v}$) is the set of splits of $u$ (resp., $v$) realizing $\str(u_{\mathcal{D}^u})$ (resp., $\str(v_{\mathcal{D}^v})$) as a proper substring of $\str(w_{\mathcal{D}})$.

Assume that both of these strings are nonempty. Then there exist arrows $\alpha, \beta \in {Q}_1$ such that  $w = u\alpha w^\prime\beta^{-1} v.$ Observe that $u\alpha w^\prime \in \mathcal{D}$. It follows that $M(\str(u\alpha w^\prime_{\mathcal{D}(u\alpha w^\prime)}))$ is the cokernel of the inclusion $M(\str(v_{\mathcal{D}^v}))\hookrightarrow M(\str(w_\mathcal{D}))$. Now we have the following extensions: 
$$0\to M({\str}(u\alpha w^\prime_{\mathcal{D}(u\alpha w^\prime)})) \to M(\str(w_{(\mathcal{D}\backslash\{u\alpha w^\prime, v\})\sqcup \{v\}})) \to M({\str}(v_{\mathcal{D}^v})) \to 0$$
and
$$0\to M({\str}(v_{\mathcal{D}^v})) \to M(\str(w_{(\mathcal{D}\backslash\{u\alpha w^\prime, v\})\sqcup \{u\alpha w^\prime\}})) \to M({\str}(u\alpha w^\prime_{\mathcal{D}(u\alpha w^\prime)})) \to 0.$$
Since $M({\str}(u\alpha w^\prime_{\mathcal{D}(u\alpha w^\prime)})), M({\str}(v_{\mathcal{D}^v})) \in \widetilde{\mathfrak{W}}$, we know that both middle terms of these extensions belong to $\widetilde{\mathfrak{W}}$. 


The proof when only one of $\text{str}(u_{\mathcal{D}^u})$ and $\text{str}(v_{\mathcal{D}^v})$ is nonempty is similar so we omit it. In each case, we contradict that $w$ appears in exactly one label in $W$, which verifies 3).

Lastly, if $w^\prime \in J(w_\mathcal{D})\backslash \left(\{w\} \sqcup \mathcal{D}\sqcup \bigcup_{u\in\mathcal{D}}S(u,\mathcal{D})\right),$ then we have that $w^\prime = u^1\alpha_1^{\pm 1}u^2\cdots u^{\ell-1}\alpha_{\ell-1}^{\pm 1}u^\ell$ for some $u^1,\ldots, u^\ell \in \{w\} \sqcup \mathcal{D}\sqcup \bigcup_{u\in\mathcal{D}}S(u,\mathcal{D})$ with $\ell \ge 2$ and some $\alpha_1,\ldots, \alpha_{\ell-1} \in Q_1.$ However, such an expression for $w^\prime$ contradicts 2), which we have already verified. We conclude that 3) holds.

The final assertion follows from Figure~\ref{polygons_fig}.\end{proof}

\begin{lemma}\label{lemma_W_in_psiB}
The indecomposable objects of $\mathfrak{W}$ are exactly the string modules defined by strings that may be realized as a concatenation of some of the strings in $\{\text{str}(w^i_{\mathcal{D}^i})\ | \ w^i_{\mathcal{D}^i} \in \text{Sim}(W)\}.$ Consequently, $W \subset \psi(B).$
\end{lemma}
\begin{proof}
Let $M(\text{str}(w_\mathcal{D})) \in \mathfrak{W}$ where $w_\mathcal{D} \not \in \text{Sim}(W).$ We induct on the length of the string $w$. 

By assumption, there exists $w_{\mathcal{D}^\prime} \in W$ with $\mathcal{D} \neq \mathcal{D}^\prime.$ This implies that there exists $u \in \mathcal{D}$ and $u^\prime \in \mathcal{D}^\prime$ such that $w = u\alpha^{\pm 1} u^\prime$ for some $\alpha \in Q_1$. Using Remark~\ref{splitting_remark}, there is a homomorphism $f: M(\text{str}{(w_\mathcal{D})}) \to M(\text{str}(w_{\mathcal{D}^\prime}))$ with $\text{im}(f) = M(\text{str}(u_{\mathcal{D}(u)}))$ and $\text{coker}(f) = M(\text{str}(u^\prime_{\mathcal{D}^\prime(u^\prime)})).$ Therefore, $M(\text{str}(u_{\mathcal{D}(u)})), M(\text{str}(u^\prime_{\mathcal{D}^\prime(u^\prime)})) \in \widetilde{\mathfrak{W}}$. By induction, each of these strings defining these modules are concatenations of a subset of the strings in $\{\text{str}(w^i_{\mathcal{D}^i})\ | \ w^i_{\mathcal{D}^i} \in \text{Sim}(W)\}$. Since $M(\text{str}(w_\mathcal{D})) \in \mathcal{M}$, these two subsets are disjoint and so $M(\text{str}(w_\mathcal{D}))$ is also a concatenation of a subset of the strings in $\{\text{str}(w^i_{\mathcal{D}^i}) \ | \ w^i_{\mathcal{D}^i} \in \text{Sim}(W)\}$.

The final assertion is now implied by Figure~\ref{polygons_fig}.
\end{proof}

\begin{proof}[Proof of Theorem~\ref{thm_psi_shad_isom}]
Lemma~\ref{lemma_Psi_to_shad} shows that the map in the statement of the Theorem is order-preserving and its image lies in $\widshad(\Pi(A))$. 

The map $\widshad(\Pi(A)) \to 2^\mathcal{S}$ defined before the statement of Lemma~\ref{lemma_shad_simples} is clearly order-preserving. That this map produces an element of $\Psi(\Bic(A))$ follows from Lemma~\ref{lemma_psiB_in_W} and Lemma~\ref{lemma_W_in_psiB}.

It is clear that these maps are inverses of each other.
\end{proof}

The following corollary is a consequence of Theorem~\ref{thm_psi_shad_isom} and Lemma~\ref{widshad-epi}.

\begin{corollary}
The poset $\Psi(\Bic(A))$ is a lattice.
\end{corollary}

\begin{proof}[Proof of Theorem~\ref{Thm_torshad_widshad_bij}]
From Theorem~\ref{Bic-Torshad} and Theorem~\ref{thm_psi_shad_isom}, we have that $\vartheta\circ B(-): \torshad(\Pi(A))\to \widshad(\Pi(A))$ and $\widetilde{\mathfrak{T}}\circ\vartheta^{-1}(-): \widshad(\Pi(A))\to \torshad(\Pi(A))$ are inverse bijections.
\end{proof}

\begin{remark}
Let $B \in \Bic(A)$ with $\lambda_{\downarrow}(B) = \{w^i_{\mathcal{D}^i}\}_{i = 1}^k$. Now let $\widetilde{\mathfrak{W}} \in \widshad(\Pi(A))$ and $\widetilde{\mathfrak{T}} \in \torshad(\Pi(A))$ denote the wide shadow and torsion shadow corresponding to $B$. It is straightforward to show that $$\vartheta\circ B(\widetilde{\mathfrak{T}}) = \text{filt}(\text{add}(\bigoplus_{i = 1}^k M(\str(w^i_{\mathcal{D}^i})) )) \cap \mathcal{M}$$ and $$\widetilde{\mathfrak{T}}\circ\vartheta^{-1}(\widetilde{\mathfrak{W}}) = \text{filt}(\text{gen}(\bigoplus_{i = 1}^k M(\str(w^i_{\mathcal{D}^i})) )) \cap \mathcal{M}.$$
\end{remark}

$$ $$

\bibliography{torshad_bib.bib}
\bibliographystyle{plain}

\end{document}